\documentclass[12pt,leqno,a4paper]{amsart}
\usepackage{amssymb,enumerate}
\newcommand{\FF}{{\mathbb{F}}}

\newcommand{\fS}{{\mathfrak{S}}}

\newcommand{\Char}{{\operatorname{char}}}
\newcommand{\Aut}{{\operatorname{Aut}}}
\newcommand{\GL}{{\operatorname{GL}}}
\newcommand{\PGL}{{\operatorname{PGL}}}

\newcommand{\SL}{{\operatorname{SL}}}
\newcommand{\Sp}{{\operatorname{Sp}}}
\newcommand{\SO}{{\operatorname{SO}}}
\newcommand{\GO}{{\operatorname{GO}}}

\let\al=\alpha
\let\sg=\sigma

\newtheorem{thm}{Theorem}[section]
\newtheorem{lem}[thm]{Lemma}
\newtheorem{cor}[thm]{Corollary}
\newtheorem{prop}[thm]{Proposition}

\theoremstyle{definition}
\newtheorem{exmp}[thm]{Example}

\theoremstyle{remark}

\begin{document}

\title[Products and commutators]{Products and commutators\\ of classes
  in algebraic groups}

\date{\today}

\author{Robert Guralnick}
\address{Department of Mathematics, University of
  Southern California, Los Angeles, CA 90089-2532, USA}
\makeatletter\email{guralnic@usc.edu}\makeatother
\author{Gunter Malle}
\address{FB Mathematik, TU Kaiserslautern, Postfach 3049,
  67653 Kaisers\-lautern, Germany}
  \makeatletter\email{malle@mathematik.uni-kl.de}\makeatother

\thanks{The first author was partially supported by the NSF
  grant DMS-1001962 and the Simons Foundation Fellowship 224965.
  The second author gratefully acknowledges financial support by ERC
  Advanced Grant 291512.}



\dedicatory{Dedicated to the memory of Tonny Springer}

\begin{abstract}
We classify pairs of conjugacy classes in almost simple algebraic groups
whose product consists of finitely many classes. This leads to several
interesting families of examples which are related to a generalization of
the Baer--Suzuki theorem for finite groups. We also answer a question of
Pavel Shumyatsky on commutators of pairs of conjugacy classes in simple
algebraic groups. It turns out that the resulting examples are exactly those
for which the product also consists of only finitely many classes.
\end{abstract}

\maketitle


\section{Introduction}  \label{sec:intro}

The Baer--Suzuki theorem asserts that in a finite group $G$, if $x \in G$ is
such that $\langle x, x^g \rangle$ is a $p$-group for all $g\in G$, then
$\langle x^G \rangle$ is a normal $p$-subgroup of $G$.

We were recently informed by Bernd Fischer that Reinhold Baer had asked what
one can say if, given $x, y\in G$, we have that $\langle x, y^g \rangle$ is a
$p$-group for all $g \in G$.  Examples in Guralnick--Malle--Tiep \cite{GMT}
show that there is not too much to say in general.   However, 
with some extra hypothesis \cite{GMT, GMBaer}, there are generalizations
along these lines.  

Here we consider the corresponding question for almost simple algebraic groups.
Our main
result is the following characterization, part of which is an analogue
of the corresponding result in the connected case \cite[Cor.~5.14]{GMT}:

\begin{thm}   \label{thm:main}
 Let $G$ be an almost simple algebraic group over an algebraically closed
 field, and let $C,D$ be $G^{\circ}$-classes.
 The following are equivalent:
 \begin{enumerate}[\rm(i)]
  \item $CD$ is a finite union of $G^\circ$-conjugacy classes.
  \item $C_{G^\circ}(x_1)\backslash G^\circ /C_{G^\circ}(x_2)$ is finite for all
   $(x_1,x_2)\in C\times D$.
  \item $\langle x_1,x_2 \rangle$ normalizes some Borel subgroup of
   $G^\circ$ for every $(x_1,x_2)\in C\times D$.
 \end{enumerate}
 Moreover, any of the above conditions implies:
 \begin{enumerate}[\rm(iv)]
  \item $[C,D]:=\{[x,y] \mid (x,y) \in C\times D\}$ is a finite union of
   $G^\circ$-conjugacy classes.
 \end{enumerate}
\end{thm}

Here, an \emph{almost simple algebraic group} is a possibly disconnected
linear algebraic group whose connected component $G^{\circ}$ is simple
and such that $C_G(G^\circ)=Z(G^\circ)$. See Theorem~\ref{thm:mainlong} for
further equivalent conditions.
Note that the result holds for $G$-classes as well as long as we avoid
the case $G=D_4.\fS_3$ (because in the disconnected cases one of the classes
is outer and so the centralizer intersects all cosets except in the excluded
case). 

While some implications in this statement have general proofs, others rely on
our explicit classification of all pairs of classes with the stated
properties.

Thus, we first classify in Theorem~\ref{thm:mainA} all pairs $C, D$ of
unipotent classes in (disconnected)
almost simple algebraic groups such that every pair in $C \times D$ generates
a unipotent group (the connected case was done in \cite{GMT}).  In fact,
it turns out to be equivalent to the condition that $CD$ consists of unipotent
elements.

We also extend the result of \cite[Thm.~1.1]{GMT} to the disconnected case by
classifying pairs $C, D$ of conjugacy classes in (disconnected) almost simple
groups such that $CD$ is a finite union of conjugacy classes.

It should be noted that for connected groups examples occur if and only if
the group is of non-simply laced type, while for disconnected groups examples
exist if and (clearly) only if the type is simply laced. See
Proposition~\ref{prop:rel} for some relation between these two types of
examples. Let's point out the following consequence of our classification:
whenever $CD$ is a finite union of conjugacy classes, then
at least one of $C,D$ contains a quasi-central class in its closure
(recall that a quasi-central class is the smallest conjugacy in a given
coset of $G^\circ$). It would be nice to have an a priori proof of this.
\medskip

Two corollaries of the results and proofs are (extending the results of
\cite{GMT} to the disconnected case): 

\begin{cor}   \label{cor:finite}
 Let $G$ be an almost simple algebraic group over an algebraically closed
 field.   Suppose that $C$ and $D$ are conjugacy classes of $G$.  Then
 either $CD$ is the union of at most $3$ conjugacy classes or $CD$ contains
 infinitely many conjugacy classes. 
\end{cor}

In the previous  result, if we replace $G$-classes by $G^\circ$-classes, then 
$5$ classes suffice.  

\begin{cor}   \label{cor:gopal}
 Let $G$ be an almost simple algebraic group over an algebraically closed
 field $k$. If $x\in G\setminus Z(G^\circ)$, then
 $\dim (C_G(x)yC_G(x)) < \dim G$ for all $y \in G$,
 unless $G=D_4(k).\fS_3$ and $x$ is an outer involution
 with $C_G(x) = B_3(k)$ and $y$ is not in $\langle G^\circ, x \rangle$.
 \end{cor}

Note that in the excluded case, we really get counter-examples.
This last corollary in the connected case was used by Prasad \cite{gopal} 
in his study of quasi-reductive groups.  

We also answer in the case of algebraic groups a question that Pavel Shumyatsky
had asked in the finite case:

\begin{thm}   \label{thm:mainC}
 Let $G$ be an almost simple algebraic group over an algebraically closed
 field of characteristic $p \ge 0$. Let $C$ be a $G^\circ$-conjugacy class of
 $G$ outside $Z(G^\circ)$. Then $[C,C]$ is the union of infinitely many
 conjugacy classes, where as before,
 $$[C,C]:=\{[x,y] \mid x,y \in C\}.$$
\end{thm}

It is worth noting that we also observe that an old result of the first author
\cite{Gur79} answers a question of Nick Katz \cite[2.16.7]{Ka} about
commutators being transvections. This is useful in the proof of the algebraic
group result.
\medskip

Finally, we observe that it is quite easy to extend our results to almost
simple algebraic groups over infinite fields. In particular, we can
answer the following question of Diaconis.

\begin{thm}  \label{cor:compact}
 Let $G$ be a simple compact Lie group. If $C$ and $D$ are noncentral
 conjugacy classes of $G$, then $CD$ consists of infinitely many conjugacy
 classes.
\end{thm}

The paper is organized as follows. In Section~\ref{sec:closed} we collect
some auxiliary results on closed conjugacy classes in not necessarily
connected algebraic groups. We then consider products of unipotent classes in
disconnected algebraic groups in Section~\ref{sec:alggroup} and classify all
cases where the product consists of finitely many classes in
Theorem~\ref{thm:mainA}. In 
Section~\ref{sec:products} we classify arbitrary products of classes meeting
only finitely many conjugacy classes and thus prove Theorem~\ref{thm:mainB}.
In Section~\ref{sec:comm} we show that these are precisely the pairs of
conjugacy classes whose commutator consists just of finitely many classes,
see Theorem~\ref{thm:commut}. In Section~\ref{sec:DNK}, we prove the common
characterization in Theorem~\ref{thm:main} of various finiteness properties
on products, commutators and double cosets and prove
Corollary~\ref{cor:finite}. In Section \ref{sec:D4.S3}, we 
fully investigate the products of two non-trivial cosets
of order 2 in the disconnected groups of type $D_4(k).\fS_3$.
In the final section, we point out how we may extend our results to the case
of infinite fields.  

\medskip

The similar questions for finite groups are considerably harder and will be
dealt with in a forthcoming paper \cite{GMBaer}.
\medskip

We thank Pavel Shumyatsky for communicating his question to us, George McNinch
for pointing out the relevance of the results of Richardson \cite{Rich2},
and Sebastian Herpel for some comments on a preliminary version.

\section{Closed conjugacy classes in disconnected algebraic groups}
\label{sec:closed}

We'll need some information on closed conjugacy classes. In connected
reductive algebraic groups, the closed classes are precisely the semisimple
ones. In the disconnected case, there can be non-semisimple elements whose
class is closed. Following Steinberg \cite[\S9]{StEnd} we call an automorphism
of a connected reductive algebraic group \emph{quasi-semisimple} if it
normalizes a Borel subgroup and a maximal torus thereof.
We also use the following result of Spaltenstein \cite[II.2.21]{Spa}:

\begin{thm}[Spaltenstein]   \label{thm:Spa}
 Let $G$ be a reductive algebraic group over an algebraically closed field.
 Then any coset of $G^\circ$ in $G$ contains at most one closed unipotent
 class. In particular, this class is contained in the closure of all
 unipotent classes in that coset.
\end{thm}

\begin{lem}   \label{lem:quasi-semisimple}
 Let $H$ be a simple algebraic group over an algebraically closed field
 of characteristic $p>0$.  Assume that $u$ is a quasi-semisimple
 automorphism of $H$ of order $p$ and $s\in H$ is a semisimple element
 commuting with $u$. Then $us$ is also quasi-semisimple.
\end{lem}

\begin{proof}
By assumption $u$ normalizes a Borel subgroup $B$ of $H$ and a maximal
torus $T \le B$. Let $D:=C_H(u)$. By \cite[Thm.~1.8(iii)]{DM94}, $T \cap D$
is a maximal torus of $D$, $B \cap D$ is a Borel subgroup of $D$ and $D$ is
connected (see also Steinberg \cite[8.2]{StEnd} for the simply connected case).
(Alternatively, this can be seen by inspection of the various cases.)
Thus, $s$ is conjugate in $D$ to an element of $T \cap D$ and so we may
assume that $s \in T \cap D$. Then $us$ normalizes $T$ and $B$.
\end{proof}

\begin{thm}   \label{thm:closedclasses}
 Let $G$ be a reductive algebraic group over an algebraically closed field $k$
 of characteristic $p$ and let $C\subset G$ be a conjugacy class, $g\in C$.
 Then the following are equivalent:
 \begin{enumerate}[\rm(i)]
  \item $C$ is closed;
  \item $C_G(g)$ is reductive; and
  \item $g$ is quasi-semisimple.
 \end{enumerate}
\end{thm}

\begin{proof}
The equivalence of (i) and (iii) is shown in \cite[Lemme~II.1.15]{Spa} (see
also \cite[Cor.~II.2.22]{Spa}).

Now assume (i), so $C$ is closed. Then $C$ is an affine variety (since $G$ is
affine). It follows by \cite[Lemma~10.1.3]{Rich2} that $C_G(g)$ is reductive.

Finally let's show that (ii) implies (iii). Note that all the conclusions
concern only $\langle g, G^\circ\rangle$ and so we assume that
$G=\langle g,G^\circ\rangle$. There is no loss in assuming that $G^\circ$ is
semisimple with $C_G(G^\circ)=1$. Moreover, we may assume that $g$ acts
transitively on the simple components of $G^\circ$. If there are $m$
components, then $C_{G^\circ}(g) \cong C_L(g^m)$ where the components are
isomorphic to $L$. Moreover, $g$ is quasi-semisimple if and only if $g^m$ is.
Thus, it suffices to assume that $G^\circ$ is simple.

Assume that $g$ is not quasi-semisimple. We show that $C_G(g)$ is not
reductive.

Write $g=su=us$ where $s$ is semisimple and $u$ is unipotent. Let $v$ be a
generator of $\langle u \rangle \cap G^\circ$. If $v \ne 1$, then by the
Borel--Tits theorem, the unipotent radical $U$ of $C_{G^\circ}(sv)$ is
nontrivial, whence since $u$ acts as a $p$-element on $U$, $C_U(u)^{\circ}$
is a positive dimensional subgroup contained in the unipotent radical
of $C_G(g)$. So $v=1$, and hence $u^p=1$.

If $p \ne [G:G^\circ]$, then $g=s$ is semisimple and in particular
quasi-semisimple. Thus we may assume that $p = [G:G^\circ]$ (in particular,
$p=2$ or $3$).

By the discussion above, we have that $u$ is an automorphism of order $p$
and $s \in G^\circ$. Note that $u$ does not lie in the unique class of outer
unipotent elements which are quasi-semisimple, since otherwise by the
previous lemma, $g$ would also be quasi-semisimple.

We consider the various cases.
For the moment, exclude the case that $p=2$ and $G^\circ=D_n(k)$, $n\ge 4$.
If $p=2$,  it  follows by \cite[19.7--19.9]{AS} that $u=vt$ where $v$ is a
quasi-semisimple graph automorphism and $t$ is a long root element in
$C_{G^\circ}(v)$.  But then, $C_{G^\circ}(u) < C_{G^\circ}(v)$ has a nontrivial
unipotent radical.
The same argument applies for $p=3$ and $G=D_4(k)$ \cite[4.9.2]{GLS}.

Finally, consider the case that $G^\circ=D_n(k)$, $n \ge 4$ and $p=2$.
Let $V$ be the natural $2n$-dimensional orthogonal module for $G$.

Suppose that $s=1$.  Note that $u$ is not a transvection on $V$ (which is the
unique quasi-semisimple unipotent class in $uG^{\circ}$). Thus, the fixed
space $W$ of $u$ has dimension at most $2n-2$.
Let $R$ be the radical of $W$ (with respect to the alternating form preserved
by $G$). Note that $R \ne 0$ for otherwise, $W$ is nonsingular and so
$V = W \oplus W^{\perp}$, a contradiction since $u$ will have fixed points
on $W^{\perp}$.   If $R$ is totally singular with respect to the quadratic
form preserved by $G$, then the centralizer of $u$ is contained in the
parabolic subgroup stabilizing $R$ and so $C_Q(u)$ is a positive dimensional
subgroup of the unipotent radical of $C_{G^{\circ}}(u)$, which thus is not
reductive.
Indeed the same argument applies if $\dim R > 1$, for then replace $R$
by $R'$, the radical with respect to the quadratic form.  So $R$ is
$1$-dimensional and is generated by a non-isotropic vector.  The stabilizer
of this $1$-dimensional space is   $\langle  x \rangle \times \Sp_{2n-2}(k)$
with $x$ a transvection. If $u \notin \langle  x \rangle$, then
$C_{G^{\circ}}(u) \cong C_{\Sp_{2n-2}(k)}(y)$ for some nontrivial unipotent
element $y$ and so is not reductive.

If $s \ne 1$, write $V = W \perp W'$ where $W=[s,V]$. Note that $W \ne 0$ since
$s\ne1$. Also, (since $p=2$) the centralizer of $s$ in the stabilizer of $W$
is contained in $\SO(W)$. In particular, if $u$ acts nontrivially on $W$,
it follows (by inspection or Borel--Tits) that the centralizer of $G$ has
non-trivial unipotent radical on $W$ and so on $V$. Thus, $u$ is trivial on $W$.
Since the centralizer of $u$ on $W'$ is reductive, it follows  by the previous
case that $u$ induces a transvection on $W'$, whence $u$ is a transvection.
This contradiction completes the proof.
\end{proof}

Note that the proof actually also shows that:

\begin{cor} \label{cor:unique}
 Let $G$ be a reductive algebraic group over an algebraically closed field.
 Let $g\in G$ with semisimple part $s$. There exists a unique closed conjugacy
 class in $gG^\circ$ whose elements have semisimple part conjugate to $s$.
\end{cor}

We also have the obvious corollary.

\begin{cor}   \label{cor:coprime}
 Let $G$ be a reductive algebraic group over an algebraically closed field $k$.
 Assume that the order of $g$ in $G/G^\circ$ is prime to the characteristic
 of $k$. Then the class of $g$ is closed if and only if $g$ is semisimple.
\end{cor}

\begin{proof}
In this situation, quasi-semisimple elements are semisimple (see
\cite[\S9]{StEnd}), and the result follows.
\end{proof}

\section{Pairs of unipotent classes in almost simple algebraic groups}   \label{sec:alggroup}

In this section we classify pairs of unipotent conjugacy classes in almost
simple algebraic groups such that all pairs of elements generate a unipotent
subgroup (see Theorem~\ref{thm:mainA}). Note that the connected case was
already treated in \cite[Thm.~1.1]{GMT}. First note the following:

\begin{lem}   \label{lem:dense}
 Let $G$ be a reductive algebraic group and $C_1,C_2,D\subset G$ three
 $G^\circ$-orbits under conjugation. Assume that
 $D\subseteq C_1C_2$ and $\dim C_1+\dim C_2=\dim D$. Then:
 \begin{enumerate}[\rm(a)]
  \item $D$ is open dense in the product $C_1C_2$;
  \item $\{(c_1,c_2)\in C_1 \times C_2 \mid c_1c_2\in D\}$ is a dense open
   subset of $C_1 \times C_2$;
  \item $C_1C_2$ consists of finitely many classes.
 \end{enumerate}
\end{lem}

\begin{proof}
Since $G^\circ$ is connected, the $G^\circ$-orbits $C_1,C_2,D$ are irreducible
and then so is $C_1C_2$. Now $\dim C_1C_2\le \dim C_1+\dim C_2=\dim D$ and
$D\subset C_1C_2$ imply that $\dim C_1C_2=\dim D$. Thus, their closures agree.
Since any two dense subsets intersect non-trivially, $D$ is the unique dense
class in $C_1C_2$, so it must be open. The second statement now follows.
As the closure of any class consists of finitely many classes, we have~(c). 
\end{proof}

We now collect some actual examples:

\begin{exmp}   \label{ex:graph-trans}
Let $S=\SL_{2n}(k)$, $n\ge 2$, with $k$ an algebraically closed field.
Let $x \in S$ be a transvection, $y\in \Aut(S)$ a graph automorphism with
centralizer $\Sp_{2n}(k)$ and $G=\langle S,y\rangle$.
Let $U$ be a root subgroup containing $x$. Let $N=N_S(U)$. Since $\Sp_{2n}(k)$
has two orbits on pairs consisting of a $1$-space and a hyperplane containing
it we see that $|\Sp_{2n}(k)\backslash S/N| = 2$. Thus, $x^S \times y^S$
consists of two $S$-orbits. One orbit is the set of commuting pairs.
When $k$ has characteristic $p>0$, reducing to the case of $n=2$ shows that
in the other orbit the product will have order $2p$. In particular, if $p=2$,
$\langle x, y^g \rangle$ is either elementary abelian of order $4$ or a
dihedral group of order $8$ for all $g\in S$.
\end{exmp}

\begin{exmp}   \label{ex:ortho p=2}
Let $G=\GO_{2n}(k)$, $n\ge 4$, with $k$ an algebraically closed field of
characteristic~$2$. Then $G$ fixes a non-degenerate alternating form $(\,,\,)$
on $V=k^{2n}$, so we may view $G$ as a subgroup of the isometry group
$\Sp_{2n}(k)$ of this form. Let $x\in G$ be a transvection, and $y$ an
involution with $(yv,v)=0$ for all $v\in V$. Then by \cite[Ex.~6.6]{GMT},
the order of $xy$ is either~2 or~4, so $x,y$ always generate a 2-group.
\end{exmp}

\begin{exmp}   \label{exmp:D4}
Let $k$ be an algebraically closed field of characteristic~3 and
$G=\SO_8(k).3$, the extension of the simple algebraic group
$G^\circ=\SO_8(k)$ by a graph automorphism of order~3. Let $C_1$ be the class
of long root elements in $G^\circ$, with centralizer of type $3A_1$, and $C_2$
the class of the graph automorphism $\sg$ with centralizer $G_2(k)$ in
$G^\circ$. We claim that $C_1C_2$ consists of 3-elements. \par
Indeed, let $x_i(1)$, $1\le i\le 4$, denote long root elements for the four
simple roots, labelled so that the node in the centre of the Dynkin diagram
has label~2, and $\sg$ permutes $x_1(1),x_3(1),x_4(1)$ cyclically.
So $x:=x_2(1)\in C_1$, and it's easily seen that $y:=x_1(1)x_3(1)x_4(1)\sg$
is conjugate to $\sg$, so lies in $C_2$. The product
$xy=x_2(1)x_1(1)x_3(1)x_4(1)\sg$ lies in a class $D$, denoted $C_4$ in
\cite[I.3]{Spa}, of dimension~24 (see also the representative $u_4$ in
\cite[Tab.~8]{MDec}). Now $C_1$ has dimension~10 and $C_2$ has
dimension~14. Since $C_1$ is a single class under $G^\circ$-conjugation,
and $C_2,D$, being outer classes, necessarily are, Lemma~\ref{lem:dense}
shows that $D$ is dense in $C_1C_2$. \par
Since $\sg$ normalizes the maximal unipotent subgroup generated by the
standard positive root subgroups, $\langle x,y\rangle$ is unipotent for our
choice of elements $x\in C_1$, $y\in C_2$ above, hence for all choices. Our
claim follows.
\par
Note that $x_2(1)$ and $x_1(1)x_3(1)x_4(1)$ are long, respectively short root
elements in the centralizer $C_{G^\circ}(\sg)=G_2(k)$ of $\sg$, with product
a regular unipotent element. This thus leads to the example in $G_2(k)$ from
\cite[Ex.~6.1]{GMT}, and we see that in both cases, the generated groups in
the generic case agree (see also Proposition~\ref{prop:rel}).
\end{exmp}

\begin{exmp}   \label{exmp:E6}
Let $k$ be an algebraically closed field of characteristic~2 and $G=E_6(k).2$,
the extension of a simple algebraic group $G^\circ=E_6(k)$ of type
$E_6$ by a graph automorphism of order~2. Let $C_1$ be the class of long root
elements in $G^\circ$, with centralizer $U.A_5(k)$, where $U$ is unipotent
of dimension~21, and $C_2$ the class of the graph automorphism $\sg$ with
centralizer $F_4(k)$ in $G^\circ$. We claim that again $C_1C_2$ only contains
unipotent elements. \par
Indeed, let $x_i(1)$, $1\le i\le 6$, denote long root elements for the six
simple roots, labelled in the standard way (so that node 4 is at the centre
of the Dynkin diagram, and $\sg$ interchanges $x_3(1),x_5(1)$).
So $x:=x_4(1)\in C_1$, and it's easily seen that $y:=x_3(1)x_5(1)\sg$ is
conjugate to $\sg$, so lies in $C_2$. The product $xy=x_4(1)x_3(1)x_5(1)\sg$
lies in a class $D$, with representative $u_{15}$ in \cite[p.~160]{Spa},
of dimension~48 (see also \cite[Tab.~10]{MGreen}).
As $C_1$ has dimension~22 and $C_2$ has dimension 26, $D$ is dense in
$C_1C_2$ by Lemma~\ref{lem:dense}. Again, $\sg$ normalizes the maximal
unipotent subgroup generated by the positive root subgroups, so
$\langle x,y\rangle$ is unipotent, and hence by density this holds for all
pairs from $C_1\times C_2$. (Note that both classes have representatives in
a Levi subgroup of type $A_3.2$, where they correspond to
Example~\ref{ex:graph-trans}.)\par
As in the case of $\SO_8(k)$, we also recover the example in the connected
group $F_4(k)$ in characteristic~2 from \cite[Ex.~6.3]{GMT}, by considering
the centralizer $C_{G^\circ}(\sg)=F_4(k)$ of $\sg$. There, $x_4(1)$ is a long
root element, and $x_3(1)x_5(1)$ is a short root element. We see that the
group generated by them is isomorphic to the one in $G$.
\end{exmp}

We are now ready to prove the main result of this section. Throughout our
proofs the following result will make inductive arguments work (see
\cite[Thm.~1.2]{Gur07}):

\begin{prop}   \label{prop:hit}
 Let $G$ be reductive, $H\le G$ a closed subgroup, and $C,D$ conjugacy
 classes of $G$.
 \begin{enumerate}
  \item[\rm(a)]  If $(C \cap H)(D \cap H)$ hits infinitely many semisimple
   $H$-classes, then $CD$ hits infinitely many semisimple $G$-classes.
  \item[\rm(b)] If $H$ is reductive and $(C \cap H)(D \cap H)$ hits infinitely
   many $H$-classes, then $CD$ hits infinitely many $G$-classes.
 \end{enumerate}
\end{prop}

\begin{thm}   \label{thm:mainA}
 Let $G$ be an almost simple algebraic group of adjoint type over an
 algebraically closed
 field $k$ of characteristic~$p\ge0$ with connected component $G^\circ$.
 Let $C_1,C_2$ be non-trivial unipotent $G^\circ$-classes of
 $G=\langle C_1,C_2\rangle$ such that $xy$ is unipotent for all
 $x\in C_1$, $y\in C_2$. Then $p\in\{2,3\}$ and (up to order of $C_1,C_2$)
 either:
 \begin{enumerate}[\rm(1)]
 \item $G=G^\circ$ is connected, $C_1$ contains long root elements and one of:
  \begin{enumerate}
   \item[\rm(a)] $p=2$, $G= \Sp_{2n}(k) = \Sp(V)$ with $n \ge 2$, and $C_2$
    contains involutions $y$ with $(yv,v)=0$ for all $v$ in $V$; or
   \item[\rm(b)]  $(G,p)=(F_4,2)$ or $(G_2,3)$, and $C_2$ contains short
    root element; or
  \end{enumerate}
 \item $G$ is the extension of $G^\circ$ by a graph automorphism of the Dynkin
  diagram of order~$p$, $C_2$ is the class of this quasi-central automorphism,
  and one of:
  \begin{enumerate}
   \item[\rm(c)] $p=2$, $G=A_{2n-1}(k).2$ with $n\ge2$, $C_1$ consists of long
    root elements and $C_2$ is the class of the graph automorphism with
    centralizer $C_n(k)$;
   \item[\rm(d)] $p=2$, $G=D_n(k).2$, $C_1$ contains involutions as given in
    \textrm{(1)(a)} for $C_n(k)$ and $C_2$ consists of transvections (the
    class of graph automorphisms with centralizer $B_{n-1}(k)$);
   \item[\rm(e)] $p=2$, $G=E_6(k).2$, $C_1$ is the class of long root elements
    and $C_2$ is the class of the graph automorphism with centralizer $F_4(k)$;
    or
   \item[\rm(f)] $p=3$, $G=D_4(k).3$, $C_1$ is the class of long root elements
    and $C_2$ is the class of the graph automorphism with centralizer $G_2(k)$.
  \end{enumerate}
 \end{enumerate}
 Moreover, in all these cases, $\langle x,y\rangle$ is unipotent for all
 $x\in C_1$, $y\in C_2$.
\end{thm}

\begin{proof}
By \cite[Thm.~1.1]{GMT} we may assume that $G$ is disconnected and $C_2$, say,
lies in $G\setminus G^\circ$. So either $p=2$ and $G^\circ$ has type $A_n$
($n\ge 2$), $D_n$ ($n\ge 4$), or $E_6$, or $p=3$ and $G^\circ$ has type $D_4$.
By Examples~\ref{ex:graph-trans}--\ref{exmp:E6} above, all the cases described
in~(2)(c)--(2)(f) do give examples.

Suppose that $C_1$ also consists of outer elements.  By
Corollary~\ref{cor:unique} there is a unique outer unipotent class lying
in the closure of all other unipotent classes in a fixed coset. By taking
closures we may assume that $C_1,C_2$ are these. If $p=3$ and
$C_1=C_2$ then a computation in $\SO_8^+(3).3$ shows that the product contains
non-unipotent classes. Else, $C_2$ contains the inverses of $C_1$, and since
outer quasi-semisimple elements normalize, but do not centralize, a
maximal torus of $G^\circ$, we obtain non-trivial semisimple products.

So we may assume that $C_1 \subset G^\circ$. Suppose for the moment that $p=2$.

First consider the case that $G^\circ = A_n(k)$, $n\ge 2$. Suppose that
$C_1$ does not consist of transvections. Then by taking closures, we may
assume that $C_1$ either consists of elements with one Jordan block of size~$3$
or two Jordan blocks of size~$2$, and $C_2$ consists of quasi-semisimple
elements. Thus, it suffices to work in $A_2(k)$ or $A_3(k)$. Note that all
inner unipotent classes and all classes of involutory graph automorphisms
intersect $A_2(2).2 \cong \PGL_2(7)$ resp. $A_3(2).2 \cong \fS_8$ and we
obtain a contradiction. So $C_1$ consists of transvections. If $n$ is even,
we may assume that $C_2$ contains quasi-semisimple elements and reduce to
the case of $A_2(k).2$. Computing in $A_2(2).2$, we see that there are no
examples. If $n\ge 3$ is odd, by \cite{Spa} all outer unipotent classes except
for the quasi-semisimple one, contain the transpose inverse graph automorphism
in their closure. So we may reduce to the case of $A_3(2).2$ to see that
$C_2$ must correspond to transpositions in the symmetric group and so must
be the class of the graph automorphisms with centralizer $C_n(k)$, the
quasi-semisimple class. We thus arrive at case~(2)(c).
 
Next consider the case that $G^\circ = D_n(k)$. Suppose that $C_2$ does
not consist of transvections in the natural $2n$-dimensional representation
of $D_n(k).2$.  By passing to closures, we may
assume that $C_2$ contains elements with exactly three nontrivial Jordan
blocks of size $2$ in the natural $2n$-dimensional representation. In
particular, we see that elements of $C_2$ leave invariant a nondegenerate
$6$-dimensional space.  Passing to the closure, we may assume that $C_1$
consists of long root elements.  In particular, we can reduce to the case of
$D_3(k) =A_3(k)$ for which we already saw that there are no such examples.

So we may assume that $C_2$ consists of transvections.  We claim that $C_1$
contains involutions $x$ so that $(xv,v)=0$ for all $v \in V$, the natural
$2n$-dimensional module for $C_n(k) \ge D_n(k).2$ as in (2)(d).
If $C_1$ consists of involutions which do not have that property, then any
$x \in C_1$ leaves invariant a nondegenerate $2$-dimensional space on which
it acts as a transvection. Thus, $xy$ ($y\in C_2$) can have odd order on that
$2$-dimensional space.

If $C_1$ does not consist of involutions, then going to the closure allows us
to assume that $C_1$ contains elements of order $4$.  It follows
that any $x \in C_1$ leaves invariant a nondegenerate subspace of dimension~$6$
or $8$ (acting as an element of order $4$).  Thus, we can reduce to either
$D_3(k)=A_3(k)$ or $D_4(k)$.  In the first case, we are done by the result
for type $A$. In the second case, we note that every class of elements of
order~$4$ in $\SO_8(k)$ intersects $\SO_8^+(2)$ and so we may compute structure
constants in $\GO_8^+(2)$ to conclude that there are no such examples.

Now suppose that $G^\circ=E_6(k)$. First suppose that elements of $C_2$ do not
have centralizer $F_4(k)$. By taking closures, we may assume that $C_1$
consists of long root elements in $E_6(k)$. By \cite[p.~250]{Spa} all outer
unipotent classes except for the quasi-semisimple one contain the graph
automorphism with centralizer $C_4(k)$ in their closure. Let $\tau$
be an involution with centralizer $F_4(k)$. Let $x \in F_4(k)$ be a long
root element in $E_6(k)$ such that $\tau x \in C_2$. Then by the result for
the connected group $F_4(k)$ there is $y\in C_1\cap F_4(k)$ with $xy$ of odd
order.

So we may assume that elements of $C_2$ have centralizer $F_4(k)$. If $C_1$
consists of long root elements then we are in case (2)(e). Else, by taking
closures, we may assume that $C_1$ is the class $2A_1$ \cite[p.~247]{Spa}.
This class has representatives in the Levi subgroup $A_5(k)$ and so the
result follows by that case.

Finally, consider the case $p=3$ with $G^\circ = D_4(k)$. Suppose that
elements of $C_2$ do not have centralizer $G_2(k)$.  By taking closures, we
may assume that $C_1$ consists of long root elements. Let $\tau$ be the graph
automorphism with centralizer $G_2(k)$. Choose $x, y \in G_2(k)$ which are
long root elements of $D_4(k)$ such that $xy$ is not a $3$-element (by the
result for connected groups) and $\tau x \in C_2$.
Then $\tau x y$ is not a $3$-element either. So we may assume that
elements of $C_2$ have centralizer $G_2(k)$. Note that the graph automorphism
fuses the three non-trivial unipotent classes of $G^\circ$ of second smallest
dimension (see \cite[p.~239]{Spa}), so if $C_1$ is not as in case~(2)(f),
then by closure we may assume that $C_1$ is one of these classes. But
computation of structure constants in the finite subgroup $\SO_8^+(3).3$
shows that this does not give an example.

Note that in all the exceptions in the statement, the group generated
by $x\in C_1$, $y\in C_2$ is always unipotent, by
Examples~\ref{ex:graph-trans}--\ref{exmp:E6}, resp. by the examples in
\cite{GMT}. This completes the proof of Theorem~\ref{thm:mainA}.
\end{proof}

\section{Products of conjugacy classes in disconnected algebraic groups}
\label{sec:products}

In \cite[Thm.~1.1]{GMT} we classified pairs of conjugacy classes $C, D$ in a
simple algebraic group such that the product $CD$ consists of a finite union
of conjugacy classes (equivalently, the semisimple parts of elements in $CD$
form a single conjugacy class). We build on the results from the previous
section to extend this to the disconnected case.

First we give some examples.

\begin{exmp}   \label{exmp:graph-refl}
Let $S=\SL_{2n}(k)$, $n\ge2$, with $k$ algebraically closed. Let $x\in S$ be
any element that is (up to scalar) a pseudoreflection, $y$ a graph automorphism
with centralizer $\Sp_{2n}(k)$ and $G=\langle S,y\rangle$. Then by
\cite[Ex.~7.2]{GMT}, there is a single $G$-orbit on $x^S \times y^S$, whence
$x^Sy^S$ is a single conjugacy class. In particular, if $k$ is not of
characteristic~2, and $x$ is a 2-element, $\langle x,y \rangle$ is a $2$-group.
\end{exmp}

\begin{exmp}   \label{exmp:GO}
Let $G=\GO_{2n}(k)$, $n \ge 4$, with $k$ algebraically closed. Suppose that
$x\in G$ is an element whose centralizer is $\GL_n(k)$ and $y\in G$ is a
reflection. As $\GL_n(k)$ acts transitively on non-degenerate 1-spaces of the
natural module for $G$, we have $G= \GO_{2n-1}(k) \GL_n(k)$ and so $G$ has a
single orbit on $x^G \times y^G$. If $k$ is not of characteristic~2 and
$x$ has finite order $m$, then $xy^g$ will have order $2m/(2,m)$.
\end{exmp}

\begin{exmp}   \label{exmp:GOtrans}
Let $G=\GO_{2n}(k)$, $n\ge 4$, with $k$ algebraically closed of characteristic
not $2$. Suppose that $x$ is unipotent with all Jordan blocks of size at most~2
and $y$ is a reflection in $G$. Then $x$ and $y$ are both trivial on a totally
singular subspace of dimension~$n-1$ of the natural module. Let $P$ denote
the parabolic subgroup stabilizing this space. Then $x$ lies in the unipotent
radical of $P$, whence $xy$ has constant semisimple part $y$. Thus we get an
example with $C_1=x^G$, $C_2=y^G$.
\end{exmp}

\begin{exmp}   \label{exmp:D4andE6}
Let $G=\SO_8(k).3$ with $k$ algebraically closed of characteristic not $3$.
Let $C_1$ be the class of root elements of $G^\circ$, and $y\in C_2$,
the class of graph automorphisms with centralizer $H=G_2(k)$ in $G^\circ$. Let
$s\in H$ be an element of order~3 with $C_H(s)=\SL_3(k)$. By looking at the
centralizer in the adjoint action on the Lie algebra of $G^\circ$ one sees that
$sy$ is $G$-conjugate to $y$. Note that $C_1\cap H$ are long root elements in
$H$. Thus by \cite[Ex.~6.2]{GMT} the product $(C_1\cap H)s$ contains an element
$us$ with $u$ regular unipotent in $C_H(s)=\SL_3(k)$, with centralizer
dimension~4 in $H$.
Thus $xsy\sim usy$ for some $x\in C_1$, with $[u,sy]=1$ and of coprime order.
It follows that $\dim C_G(xsy)=\dim C_G(usy)=\dim C_H(u)=4$, so that $C_1C_2$
contains a class of dimension $\dim G-4=24=\dim C_1+\dim C_2$. An application of
Lemma~\ref{lem:dense} shows that $C_1C_2$ is a union of finitely many classes.
\par
Similarly, let $G=E_6(k).2$ with $k$ algebraically closed of characteristic
not~$2$. Let $C_1$ be the class of root elements of $G^\circ$ and
$y\in C_2$, the class of graph automorphisms with centralizer $H=F_4(k)$
in $G$. Let $s\in H$ be an involution with $C_H(s)=B_4(k)$. By looking at the
adjoint action one sees that $sy$ is $G$-conjugate to $y$. Again, $C_1\cap H$
consists of long root elements in $H$. Thus by \cite[Ex.~6.4]{GMT} the
product $(C_1\cap H)s$ contains elements $u\in B_4(k)$ whose square is a short
root element, with centralizer dimension~30 in $H$. Arguing as before, we see
that $C_1C_2$ contains a class of dimension $\dim G-30=48=\dim C_1+\dim C_2$,
and we
conclude by Lemma~\ref{lem:dense} that $C_1C_2$ is a finite union of classes.
\end{exmp}

\begin{exmp}   \label{exmp:D4.S3}
Let $G=\SO_8(k).\fS_3$ with $k$ algebraically closed. Let $C_1,C_2$ be
$G^\circ$-classes of graph automorphisms with centralizer $H=B_3(k)$ in
$G^\circ$ lying in two different $G^\circ$-cosets. Then $C_1,C_2$ both have
dimension~7, while any class in the coset of order~3 has dimension at
least~14, which is attained for the quasi-central class of graph automorphisms
of order~3. So the product consists of that single class by
Lemma~\ref{lem:dense}. See Section~\ref{sec:D4.S3} for more on this.
\end{exmp}

\begin{thm}   \label{thm:mainB}
 Let $G$ be an almost simple algebraic group of adjoint type over an
 algebraically closed
 field $k$ of characteristic $p\ge 0$ with connected component $G^\circ$.
 Let $C_1$ and $C_2$ be nontrivial $G^\circ$-conjugacy classes in
 $G = \langle C_1, C_2 \rangle$. Suppose that $C_1C_2$ consists of finitely
 many conjugacy classes. Then one of the following holds:
 \begin{enumerate}[\rm(1)]
  \item $C_1,C_2$ are unipotent classes and we are in one of the cases of
   Theorem~\ref{thm:mainA}; or
  \item $G=G^\circ$ is connected, $C_2$ is unipotent, and one of
  \begin{enumerate}
   \item[\rm(a)] $G= \Sp_{2n}(k)$, $n\ge2$, $p\ne2$, $C_1$ contains
    involutions and $C_2$ contains long root elements;
   \item[\rm(b)] $G=\SO_{2n+1}(k)$, $n\ge2$, $p\ne2$, $C_1$ contains
    reflections (modulo scalars) and $C_2$ contains unipotent elements
    with all Jordan blocks of size at most $2$;
   \item[\rm(c)]  $G=F_4$, $p\ne2$, $C_1$ contains involutions with
    centralizer $B_4(k)$ and $C_2$ contains long root elements;
   \item[\rm(d)]  $G=G_2$, $p\ne3$, $C_1$ contains elements of order $3$ with
    centralizer $\SL_3(k)$ and $C_2$ contains long root elements; or
  \end{enumerate}
  \item $G$ is disconnected, $C_2$ contains quasi-semisimple graph
   automorphisms, and one of
  \begin{enumerate}
  \item[\rm(e)]  $[G:G^\circ]=2$, $G^\circ = A_{2m-1}(k)$, $m>1$,
   $C_1$ contains (semisimple) pseudo-reflections and $C_2$ contains graph
   automorphisms with centralizer $C_m(k)$;
  \item[\rm(f)]  $[G:G^\circ]=2$, $G^\circ = A_{2m-1}(k)$ with $p \ne 2$, $m>1$,
   $C_1$ contains transvections and $C_2$ contains graph automorphisms with
   centralizer $C_m(k)$;
  \item[\rm(g)]  $[G:G^\circ]=2$, $G^\circ = D_n(k)$, $n \ge 4$,
   $C_1$ contains semisimple elements with centralizer $A_{n-1}(k)$ and
   $C_2$ contains graph automorphisms with centralizer $B_{n-1}(k)$;
  \item[\rm(h)]  $[G:G^\circ]=2$, $G^\circ = D_n(k)$ with $p \ne 2$, $n\ge4$,
   $C_1$ contains unipotent elements with all Jordan blocks of size at most~2
   and $C_2$ contains graph automorphisms with centralizer $B_{n-1}(k)$;
  \item[\rm(i)]  $[G:G^\circ]=2$, $G^\circ = E_6(k)$ with $p \ne 2$,
   $C_1$ contains long root elements and $C_2$ contains graph automorphisms
   with centralizer $F_4(k)$;
  \item[\rm(j)]  $[G:G^\circ]=3$, $G^\circ = D_4(k)$ with $p \ne 3$,
   $C_1$ contains long root elements and $C_2$ contains graph automorphisms
   with centralizer $G_2(k)$; or
  \item[\rm(k)]  $[G:G^\circ]=6$, $G^\circ = D_4(k)$, $C_1$ and $C_2$ are
   classes in different cosets modulo $G^\circ$ of graph automorphisms of
   order~2. 
  \end{enumerate}
 \end{enumerate}
\end{thm}

Note that cases (f), (h), (i) and~(j) are direct analogues of the unipotent
pairs in Theorem~\ref{thm:mainA}. We analyze the case 3(k) completely in
Section \ref{sec:D4.S3}, see Proposition~\ref{prop:D4.S3}. If we consider
$G$-classes instead of $G^\circ$-classes, we get the same examples except for
the ones in~3(k).

The proof of Theorem~\ref{thm:mainB} will be given in a series of five
lemmas. First note that if $G$ is connected, the result was shown in
\cite[Thm.~1.1]{GMT}. So we may and will assume that $G^\circ$ is one of
$A_n(k)$ ($n \ge 2$), $D_n(k)$ ($n \ge 4$), or $E_6(k)$. Let $r=[G:G^\circ]$.
We may assume that $C_2$ consists of outer automorphisms.
We consider various cases.

\begin{lem}   \label{lem:bothout}
 The claim in Theorem~\ref{thm:mainB} holds when neither of the classes
 $C_1,C_2$ is contained in $G^\circ$.
\end{lem}

\begin{proof}
Passing to closures, we may first assume by Theorem~\ref{thm:closedclasses}
that $C_1,C_2$ are both quasi-semisimple. Then there exist
$(x,y)\in C_1\times C_2$ normalizing a common Borel subgroup $B$ and a
maximal torus $T \le B$. Let $S:=[y,T]$. Then $yS\subseteq C_2$ and $S$ is
infinite as $y\notin C_G(T)=C_{G^\circ}(T)$. Let $m$ denote the order of
$z=xy$ in $G/G^\circ$. If $m=1$ then $xyS\subset T\cap C_1C_2$ is infinite,
so contains infinitely many classes.
\par
If $m>1$ then necessarily $G^\circ=D_4(k)$. If $m=2$ then $(xyt)^2=z^2t^zt$,
for $t\in S$, is finite only if $z$ acts as $-1$ on $S$. Note that this
configuration can only occur when $C_1$,$C_2$ lie in cosets of $G^\circ$ of
order~2, 3 respectively. Explicit computation shows that here $z$ does not act
as $-1$ on $S$.\par
Thus we have that $m=3$. If $C_1,C_2$ lie in the same coset of order~3, then
$(xyt)^3=\{z^3t^{z^2}t^zt\mid t\in S\}$. Again, explicit computation shows
that this is not finite for suitable $N_G(T)$-conjugates of $x,y$. The only
remaining possibility is that $G/G^\circ\cong\fS_3$, $x,y$ lie in two
different cosets of order~2, as in (3)(k) of the conclusion.
\end{proof}

\begin{lem}   \label{lem:wreath}
 Let $G=A_1(k)\wr Z_r$, the wreath product of $A_1(k)$ with the cyclic group
 of prime order~$r$. Let $C_1\subset G^\circ$ be a conjugacy class whose
 projection to at least two factors is non-central, and
 $C_2\subset G\setminus G^\circ$ any conjugacy class. Then:
 \begin{enumerate}
  \item[\rm(a)] $C_1C_2$ consists of infinitely many conjugacy classes.
  \item[\rm(b)] $[C_1,C_2]$ consists of infinitely many conjugacy classes.
 \end{enumerate}
\end{lem}

\begin{proof}
Let $\sigma$ denote a generator of the cyclic subgroup of $G$ permuting the
$A_1(k)$-factors. Any outer element is conjugate to one of the form
$\sigma(1,\ldots,1,x)$ for some $x\in A_1(k)$. An easy computation with such
elements shows~(a) and~(b), using that the product of any two non-central
classes of $A_1(k)$ meets infinitely many classes.
\end{proof}

\begin{lem}   \label{lem:An}
 The claim in Theorem~\ref{thm:mainB} holds when $G^\circ=A_n(k)$.
\end{lem}

\begin{proof}
By our initial reductions and Lemma~\ref{lem:bothout} we may assume that
$r=2$, $C_1\subset G^\circ$ and $C_2$ is outer.

\medskip
Case 1. $n$ is even.

\noindent
In this case, we claim there are no examples. By taking closures, we may
assume that $C_2$ contains quasi-semisimple elements. Thus $x_2\in C_2$
normalizes a maximal torus, hence any Levi subgroup $L$ of $G^\circ$ invariant
under the graph automorphism. Again by closure, $C_1$ consists of semisimple
elements or of transvections. In either case, we see that one can find
$x_i \in C_i$ normalizing a subgroup $A_2(k)$ and acting nontrivially. Thus,
it suffices to take $n=2$. Now any quasi-semisimple element of $A_2(k)$
normalizes a maximal torus, so is of the form $x t$, $t\in T$, for some
fixed quasi-semisimple element $x$. It is then a straightforward matrix
computation to see that $C_1C_2$ always meets elements with distinct
characteristic polynomial, so with distinct semisimple part.

\medskip
Case 2. $n=2m-1 \ge3$ is odd.

\noindent
Arguing as above we can reduce to $x_1,x_2$ lying in a Levi subgroup of type
$A_3$. Inspection shows that then $x_2$ has to be an involution. We claim that
$C_2$ consists of involutions with centralizer $C_m(k)$. Suppose not. Then by
\cite[Tab.~4.5.1]{GLS}, elements in $C_2$ have centralizer $D_m(k)$.
We may assume that $x_1\in C_1$ is either semisimple or a transvection, and
then after conjugation we have $x_1,x_2$ in a disconnected subgroup with
connected component a Levi subgroup of type $A_{n-1}$, where no example exists
by the first case.
\par
If $x_1$ is semisimple, there is some root of $G$ which is
non-trivial on $x_1$. Since $G^\circ$ has just one root length, after
conjugation we may assume that $x_1$ lies in an $x_2$-stable Levi subgroup $L$
of type $D_2$ but not in its center, with $x_2$ swapping the two factors.
Moreover, if no eigenspace of $x_1\in C_1$ has dimension $n-1$, then we may
arrange that the image of $x_1$ in both factors is non-trivial. In that case,
by Lemma~\ref{lem:wreath} the product $C_1C_2$ meets infinitely many semisimple
classes inside the wreath product $A_1(k)\wr2$. On the other hand, if $x_1$
is a pseudo-reflection (modulo scalars), we get case~(3)(e) of the assertion by
Example~\ref{exmp:graph-refl}. \par
If $x_1$ is unipotent, we may assume that $p>2$ since otherwise we are in the
situation of Theorem~\ref{thm:mainA}. If $x_1$ is a transvection, then we get
case~(f) by Example~\ref{ex:graph-trans}. Else, by closure we may assume that
$x_1$ has two Jordan blocks of size~2 or one of size~3. In the first case,
we may reduce to $A_3$, in the second to $A_2$, and no examples arise.
\par
If $x_1$ is neither unipotent nor semisimple, then by the previous arguments
it must be the commuting product of a pseudo-reflection with a transvection.
In this case, we may again reduce to the wreath product $A_1(k)\wr2$ and
apply Lemma~\ref{lem:wreath}.
\end{proof}

\begin{lem}
 The claim in Theorem~\ref{thm:mainB} holds when $G^\circ=D_n(k)$, $n\ge4$.
\end{lem}

\begin{proof}
As in Lemma~\ref{lem:An} we may assume that $C_1\subset G^\circ$ and $C_2$
is outer, so $r\in\{2,3\}$.

\medskip
Case 1. $r=2$.

\noindent
First suppose that $x_1$ is semisimple but not as in the conclusion~(g).
By taking closures we may assume $x_2$ is quasi-semisimple. Then we can
reduce to $\GO_4(k)=A_1(k)\wr2$. Moreover $x_1$ projects to elements which
are nontrivial in each factor $A_1(k)$ whence $C_1C_2$ contains infinitely
many classes by Lemma~\ref{lem:wreath}.
Next suppose that $x_1$ has centralizer $\GL_n(k)$. If $x_2$ is
quasi-semisimple but not a reflection, we can reduce to
$D_3(k)=A_3(k)$ and invoke Lemma~\ref{lem:An}. So we may assume that the
semisimple part of $x_2$ is a reflection. If
its unipotent part is nontrivial, then by taking closures we may assume it
is a root element in $\SO_{2n-1}(k)$ and so either has two Jordan blocks of
size~2 or one Jordan block of size~3. Again, we can reduce to $D_3(k)$ where
this cannot happen. When $x_2$ is a reflection, we arrive at~(3)(g) by
Example~\ref{exmp:GO}.

If $x_1$ is unipotent with all Jordan blocks of size at most~2 and $x_2$ is
a reflection, then this gives~(3)(h) by Example~\ref{exmp:GOtrans}. If $x_2$
does not consist of reflections, we may again reduce to $D_3(k)$ to obtain a
contradiction. If $x_1$ has a Jordan block of size at least~3, by taking
closures we may assume that it has just one Jordan block and $x_2$ is
semisimple, and then we may reduce to a subgroup $D_3(k)$ for which there are
no such examples.

Finally, suppose that $x_1$ is a mixed element. The above argument shows that
$x_2$ is a reflection, the semisimple part of $x_1$ has centralizer $\GL_n(k)$
and its unipotent part has Jordan blocks of size at most~2. Again, we can
reduce to $D_3(k)$ to rule out this case.

\medskip
Case 2. $r=3$, so $n=4$.

\noindent
First assume that $x_2\in C_2$ is quasi-semisimple. If $x_1$ is unipotent,
then $x_1,x_2$ lie in the normalizer of a maximal
parabolic subgroup $P$ of type $A_1^3$. If $x_1$ is not a long root element,
the image of $x_1$ in the Levi subgroup $L$ of $P$ is non-trivial in at least
two of the factors, whence the product $C_1C_2$ meets infinitely many classes
in $L$ by Lemma~\ref{lem:wreath}. If $x_1$ is a long root element, we may
conjugate $x_1,x_2$ into the
graph centralizer $G_2(k)$. By \cite[Thm.~5.11]{GMT} the only example there is
for $x_2$ to be the 3-element with centralizer $\SL_3(k)$. But by the argument
in Example~\ref{exmp:D4andE6}, all such elements fuse into the class of graph
automorphisms with centralizer $G_2(k)$, so we arrive at~(3)(j). \par
Any non-central semisimple element has a conjugate which is non-trivial in
at least two factors of the Levi subgroup $L$ of type $A_1^3$ and thus gives
infinitely many classes there. Finally, if $x_1$ is a mixed element, we may
assume that its unipotent part is a long root element, and then again its
image in $L$  is non-trivial in at least two factors. \par
Finally, if $x_2$ is not quasi-semisimple, by the previous argument and taking
closures, we may assume that $x_2 = xu =ux$ where $u$ is a long root element
and $x$ is quasi-semisimple of order~3, and $x_1$ is a long root element.
In particular if we take $x_1$ centralizing $x$, we see that $x_1$ and $u$
are conjugate in $G_2(k)$ and since $u^{G_2(k)}u^{G_2(k)}$
hits infinitely many classes, we are done.
\end{proof}

\begin{lem}
 The claim in Theorem~\ref{thm:mainB} holds when $G^\circ=E_6(k)$.
\end{lem}

\begin{proof}
Here we have $r=2$. By closure we may assume that $x_1\in C_1$ is semisimple
or unipotent of type $A_2$, or
with unipotent part a long root element, and $x_2\in C_2$ is quasi-semisimple.
Hence we may arrange so that $x_1,x_2$ are contained in a Levi subgroup of
type $A_5$, which is normalized by the graph automorphism. By
Lemma~\ref{lem:An} this
forces $x_2$ to be an involution. There are two classes of outer involutions
in $G$, with centralizers $F_4(k)$ resp. $C_4(k)$. The outer involutions of
$A_5(k).2$ with centralizer of type $D_3$ fuses into the outer class with
centralizer of type $C_4$ (since $A_5(k)$ has composition factors of
dimensions $6,6,15$ on the 27-dimensional module for $E_6(k)$, and so has
$D_3(k)$, while $F_4(k)$ has a 1-dimensional constituent), so by our
considerations in Lemma~\ref{lem:An}, $x_2$ must have centralizer $F_4(k)$,
and $x_1$
must be a long root element. This occurs by Example~\ref{exmp:D4andE6}.
\end{proof}

We have now discussed all possibilities for $G$ and thus completed the proof
of Theorem~\ref{thm:mainB}.

We note the following relation between the examples in disconnected groups
in Theorem~\ref{thm:mainB} and the ones in connected groups in
\cite[Thm.~1.1]{GMT}:

\begin{prop}   \label{prop:rel}
 Let $G$ be a simple algebraic group, $\tau$ a graph automorphism of $G$ with
 reductive centralizer $H=C_G(\tau)$. Suppose moreover that for some
 $1\ne g\in H$, $\tau$ and $g\tau$ are $G$-conjugate.
 If $C:=\tau^G$ and $D\subset G$ are classes such that $CD$ consists of only
 finitely many classes in $G\tau$, then the same holds for $C_1=g^H$,
 $D_1=D\cap H$ in $H$.
\end{prop}

\begin{proof}
Clearly, $C_1D_1\tau=g^H\tau D_1=(g\tau)^H D_1\subset CD$, and if the latter
meets only finitely many $G$-classes, then the former only meets finitely
many $H$-classes, see Proposition~\ref{prop:hit}.
\end{proof}

This applies to cases~(e), (f), (h), (i) and~(j) in Theorem~\ref{thm:mainB}.

\section{Commutators}   \label{sec:comm}

In this section we turn to commutators of pairs of conjugacy classes. We
first investigate the commutator of a single class and thus prove
Theorem~\ref{thm:mainC}. The proof follows the ideas in \cite[\S5]{GMT}.

It is convenient to recall the following result \cite{Gur79} (which
answered a question of Katz before it was asked --- see \cite[2.16.7]{Ka}).

\begin{thm}  \label{thm:gur}
 Let $G$ be a simple algebraic group of type $A$ or $C$. Suppose that
 $x,y \in G$ with $[x,y]$ a long root element (i.e., a transvection).
 Then $\langle x, y \rangle$ is contained in a Borel subgroup.
\end{thm}

\begin{proof}
We may as well work in $\SL_n$. Suppose that $z:=x^{-1}y^{-1}xy$ is a long
root element. Write $z = I +N$ where $N$ is a nilpotent rank one matrix. Then
$xy - yx = yxN$ is a rank one matrix. Now apply the main result
of \cite{Gur79}.
\end{proof}

\begin{proof}[Proof of Theorem~\ref{thm:mainC}]
Let $G$ be an almost simple algebraic group over $k$ and $C$ a
$G^\circ$-conjugacy
class of $G$ outside $Z(G^\circ)$. Let $w: C \times C \rightarrow G$ be the
map defined by $w(x,y)=[x,y]$.

Since $C$ is an irreducible variety, if the result fails, the semisimple
part of $[x,y]$ for $(x,y) \in C \times C$ would be constant. Of course,
$1=[x,x]$ and therefore the image of $w$ must be contained in the set of
unipotent elements of $G$. If $[C,C]$ consists of finitely many classes, the
same is true for its closure, so we may also replace $C$ by any class in its
closure and so assume that $C$ is semisimple or unipotent.

First suppose that $G \cong A_1(k)$.   We claim that it suffices
to show the result for $\SL_2(k)$. Indeed, lift $C$ to a class of $\SL_2(k)$;
the image of $w$ is irreducible and contains $1$, so still all commutators
are unipotent (not just unipotent modulo the center). 
 
Choose $(x,y) \in C \times C$ so that $x$ and $y$ are not contained in a
common Borel subgroup (since any non-central element lies in at most two Borel
subgroups, this clearly can be done). By Theorem~\ref{thm:gur},
$[x,y]$ is not unipotent.

Next suppose that $G=G^\circ$ is connected. If $C$ contains semisimple
elements, we can choose $x \in C \cap T$ where $T$ is a maximal torus and so
$x$ does not centralize some root subgroup $U$.
Let $U^-$ be the corresponding negative root subgroup. Then
$\langle T,U,U^-\rangle$ is a product of a torus with an $A_1(k)$,
whence the result follows.

If $C$ is unipotent, by closure we may assume that it consists of root
elements. Then we just work in $\langle U,U^-\rangle$ where $U$ is a root
subgroup generated by elements of $C$ and again the result follows by the
$A_1(k)$ result.

Finally suppose that $G$ is disconnected. We may assume that $C$ consists of
outer automorphisms, and by closure that it is closed. In particular, elements
of $C$ are quasi-semisimple by Theorem~\ref{thm:closedclasses}. By definition,
any such element normalizes, but does not centralize, a maximal torus of
$G^\circ$, and thus there exist non-trivial semisimple commutators of elements
in $C$.
\end{proof}

It follows easily from the previous result that if $C$ is a noncentral
conjugacy class of $\ell$-elements for some prime $\ell$, then commutators
of pairs of elements of $C$ cannot consist only of $\ell$-elements (this is
related to the question of Shumyatsky).  

If we take distinct classes, there are examples in \cite{GMT} showing that
the commutators could always be unipotent. Arguing as in \cite{GMT}, we can
actually classify all such pairs.

\begin{thm}   \label{thm:commut}
 Let $G$ be an almost simple algebraic group over an algebraically closed
 field $k$ of characteristic $p \ge 0$. Let $C,D$ be non-central
 $G^\circ$-conjugacy classes in $G$ with $G=\langle C,D\rangle$. Assume that
 $C,D$ do not lie in distinct cosets of order~2 if $G=D_4(k).\fS_3$. Then
 $$[C,D]:=\{[x,y] \mid x\in C,\,y\in D\}$$
 is the union of finitely many conjugacy classes if and only if $C,D$ are as
 given in Theorem~\ref{thm:mainA} or as in Theorem~\ref{thm:mainB}\textrm{(a)--(j)}.
 In particular, $C\ne D$ in this case.
\end{thm}

\begin{proof}
Let $w:C\times D\rightarrow G$ be as in the previous proof. By irreducibility
the semisimple part of elements in $[C,D]$ has to be constant.

We can now essentially follow the proof in \cite{GMT}, respectively the proofs
of Theorems~\ref{thm:mainA} and \ref{thm:mainB}. First assume that $G$ is
connected. For $G=\SL_2(k)$, choosing $(x,y)\in C\times D$ in the same Borel
subgroup, $[x,y]$ is clearly unipotent. On the other hand, choose $x$ and $y$
so that they are not contained in a common Borel subgroup. Then by
Theorem~\ref{thm:gur}, $[x,y]$ is not unipotent. So there's no example for
$\SL_2(k)$.

If $[C,D]$ consists of finitely many classes, the same is true for its
closure, so we may also replace $C$ and $D$ by any classes in their closures
and so assume that each is semisimple or unipotent. Moreover, if one of them
is a
unipotent class, we may assume that it consists of root elements. We can now
argue exactly as in the proof of \cite[Thm.~5.11]{GMT} to rule out every
configuration apart from those in \cite[Thm.~5.11(2)--(6)]{GMT}, except for
the case that $p\ne2$, $G=\Sp_4(k)$, and the semisimple part of $c$ is an
involution. But again by working inside the subsystem subgroup $A_1(k)^2$ we
see that $[C,D]$ can have infinitely many distinct semisimple parts by
Lemma~\ref{lem:wreath}.

In all the exceptions of \cite[Thm.~5.11(2)--(6)]{GMT}, any pair
$(x,y)\in C\times D$ lies in some common Borel subgroup of $G$ by
\cite[Cor.~5.14]{GMT}, whence the commutator $[x,y]$ is unipotent and so
$[C,D]$ consists of finitely many (unipotent) classes.
\par
Now assume that $G$ is not connected, but $C,D$ are unipotent. Then the same
would have to hold for all classes in $[C,D]$, since all unipotent elements
have conjugates in a maximal unipotent subgroup of the normalizer of a
Borel subgroup. Following the arguments in the proof of
Theorem~\ref{thm:mainA} we see that the assertion follows once we show
that the only pairs of 2-power classes in $\fS_8$, $\PGL_2(7)$, $\GO_8^+(2)$
with commutator consisting of 2-elements are those in the conclusion.
This can be checked by direct computation.
\par
Finally assume that $G$ is disconnected and not both $C,D$ are unipotent. 
Here, we go through the proof of Theorem~\ref{thm:mainB}. If both classes
are outer (and quasi-semisimple), then we may find representatives $x,y$
normalizing a maximal torus $T$, as in the proof of Lemma~\ref{lem:bothout}.
First assume that $x,y$ lie in the same cyclic subgroup of $G/G^\circ$.
Then with $S=[T,y^{-1}]$ and $R=[x,S]$ we have
$[x,Sy]=\{x^{-1}y^{-1}xry\mid r\in R\}\subseteq [C_1,C_2]\cap T$, which
is always infinite, since $x$ may be chosen not to centralize the commutator
space of $y^{-1}$. The only remaining case is when $G=D_4.\fS_3$ with
$x$, $y$ in non-trivial cosets of different order. Here, the cube of the
commutator $[x,Sy]$ hits infinitely many classes in $T$ since
$\{r\in R\mid r^{z^2}r^zr\}$ is not finite, for $z=[y^{-1},x]$.
\par
Thus, $C$ is inner, say. As in the proof of Lemma~\ref{lem:An} we are done
for $A_n(k)$ once we have shown the claim for $n=2$ and $n=3$, where it is a
direct matrix calculation. The arguments for all other types go through
unchanged.
\end{proof}

In the case of commutators between $G^\circ$-classes in two different cosets
of order~2 in $D_4(k).\fS_3$, there are many further examples, see
Example~\ref{exmp:commD4} below.

\section{Double cosets, products and commutators}   \label{sec:DNK}

The aim of this section is the Characterization Theorem~\ref{thm:main} for
various finiteness properties. For this we need to investigate some of the
examples from the previous three sections a bit more closely.

\begin{prop}   \label{prop:orb1}
 Let $G=\GO_{2n}(k)$, $n\ge3$, with $k$ algebraically closed of
 characteristic~$p$. Let $x\in G^\circ$ be unipotent such that either
 \begin{enumerate}[\rm(1)]
  \item $p\ne 2$ and all Jordan blocks of $x$ have size at most~$2$; or
  \item $p=2$, $x^2 =1$ and $(xv,v)$=0 for all $v$ in the natural module $V$
   for $G$ (under the associated alternating form).
 \end{enumerate}
 Then $C_G(x)$ has at most two orbits on nondegenerate 1-spaces in $V$.
\end{prop}

\begin{proof}
We claim that any nondegenerate 1-space is conjugate under $C_G(x)$ to an
element in a nondegenerate 4-space or 6-space that is $x$-invariant.

Write $V = V_1 \perp V_2$  where $V_1$ and $V_2$ are $x$-invariant and $x$ is
trivial on $V_1$ and has all Jordan blocks of size exactly $2$ on $V_2$.

It suffices to deal with each space separately. In the first case the
centralizer contains $\GO(V_1)$ and we can move any 1-space into any
nondegenerate 2-space. So we need to consider $V_2$. We claim that any vector
is $C_G(x)$-conjugate to an element of some $x$-invariant nondegenerate 4-space.

In both cases the fixed space of $x$ contains a maximal totally singular
isotropic space. Thus we can write $V_2 = U\oplus U'$ where $U$ and $U'$ are
maximal complementary totally isotropic spaces, with $U$ the fixed space
of $x$. Thus $x$ is trivial on $V_2/U$ as well and so $x$ is in the
unipotent radical $Q$ of the stabilizer of $U$. Then $x$ corresponds to a
full rank skew symmetric matrix in $Q$. Thus, the centralizer of $x$ is
$Q.\Sp(U)$ with $\Sp(U)$ acting isomorphically on both $U$ and $U'$.
Since $\Sp(U)$ is transitive on non-zero vectors of $U$, it is easy to see
that given $v = u + u'$, we can move $u,u'$ into a nondegenerate $x$-invariant
4-space.

Thus we are reduced to $\GO_4(k)$ and $\GO_6(k)$. It is straightforward to
check that in the first case there is a unique orbit, while in the second
case there are two orbits (as can be seen for example by checking that over
finite fields $\FF_q$, there are two orbits, of lengths
$q^3(q^2-1),q^2(q-1)$).
\end{proof}


\begin{prop}   \label{prop:orb2}
 \hskip 1pc{ }
 \begin{enumerate}
  \item[\rm(a)] Let $G=D_4(k).3$, the extension by a graph automorphism, and
   $x\in G^\circ$ a long root element. Then
   $|G_2(k)\backslash G^{\circ}/C_{G^{\circ}}(x)|=5$
   and $|G_2(k)\backslash G/C_G(x)|=3$.
  \item[\rm(b)] Let $G=E_6(k).2$, the extension by a graph automorphism, and
   $x\in G^\circ$ a long root element. Then $|F_4(k)\backslash G/C_G(x)|=2$.
 \end{enumerate}
\end{prop}

\begin{proof}
First consider $G=D_4(k).3$, the extension by the graph automorphism $\sg$ of
order~3. Let $B$ be a $\sg$-stable Borel subgroup of $G^\circ$. Let $P$ be the
normalizer in $G^\circ$ of a $\sg$-stable root subgroup, a maximal parabolic
subgroup of type $A_1^3$. Its derived subgroup $P'$ is the centralizer of a
long root element in $G^\circ$. In order to show that $G_2(k)\backslash G/P'$
is finite, it suffices to see that $G_2(k)$ has finitely many orbits on the
set of long root elements of $G^\circ$. We claim that there are precisely five
orbits. Clearly, there's one orbit of long root elements (in $B$) centralized
by $\sg$; these are long root elements in $H=C_{G^\circ}(\sg)\cong G_2(k)$ and
thus have centralizer $Q'$ in $H$, where $Q$ is a maximal parabolic subgroup
of $H$.
Furthermore, there are three orbits of long roots fused under $\sg$, such
that the product over any $\sg$-orbit is a short root element in $H$. These
have centralizer $U_3.A_1(k)$, with $U_3$ unipotent of dimension~3.
Finally, let $x\in U=R_u(B)$ be the product of a root element with support~2
(on the simple roots of $G^\circ$ with respect to $B$) with a commuting root
element with support~3 . Then the centralizer of $u$ in $U\cap H$ has
dimension~4. The product $uu^\sg u^{\sg^2}$ is a unipotent element
($x_{a+b}(1)x_{2a+b}(1)$ with respect to the simple roots $a,b$ of $H$)
with 4-dimensional unipotent centralizer in $H$. Thus $\dim(C_H(u))=4$.
(This is the element constructed in Example~\ref{exmp:D4andE6} when $p\ne3$,
respectively in Example~\ref{exmp:D4} when $p=3$.)  \par
For $p>0$, over the field with $q=p^a$ elements the corresponding
$G_2(q)$-orbits in $D_4(q).3$  have lengths $q^6-1$, $q^2(q^6-1)$ (3 times),
and $q^2(q^2-1)(q^6-1)$, which adds up to the number of long root elements,
so we've accounted for all orbits in this case. By arguing as in
\cite[Ex.~6.2]{GMT} we conclude that the claim holds over any algebraically
closed field. Note that since $\sg$ permutes the three orbits of dimension~8,
we obtain three orbits under $C_G(\sg)$.
\par
Now consider $G=E_6(k).2$, the extension by the graph automorphism $\sg$ of
order~2. The centralizer of a long root element in $G^\circ$ is the derived
subgroup $P'$ of a parabolic subgroup $P$ of $G^\circ$ of type $A_5$. The group
$H=C_{G^\circ}(\sg)\cong F_4(k)$ has one orbit on long root elements contained
in $H$, so centralized by $\sg$, with centralizer $Q'$, where $Q$ is an
end-node parabolic subgroup of $H$. Now let $x$ be a long root element not
fixed by $\sg$, corresponding to the root $\al=\sum_{i=1}^6 a_i\al_i$, with
$\al_1,\ldots,\al_6$ the simple roots of $E_6$ in the standard numbering and
$(a_1,\ldots,a_6)=(1,1,1,2,2,1)$. This root subgroup is normalized by a maximal torus of $H$,
and also centralized by all but three positive root subgroups in $H$. Moreover,
$\al$ is centralized by a subgroup of type $W(A_3)$ in the Weyl group of $H$.
Thus $C_H^\circ(x)=U.A_3(k)$ with $U$ unipotent of dimension~15, and the
$H$-orbit of $x$ has dimension~$22=\dim E_6(k)-\dim P'$, so it is the dense
orbit. Note that the $\sg$-conjugate root element $x^\sg$ (in the root
subgroup corresponding to $(112211)$) has exactly the same connected
centralizer in $H$, and for $p\ne2$, $C_H(x x^\sg)=U.A_3(k).2$ is an extension
of $C_H^\circ(x)$ of degree~2 (see \cite{Sh74}). Since $W(F_4)$ contains an
element interchanging the root subgroups of $x$ and $x^\sg$, we deduce that
$C_H(x)=U.A_3(k)$ is connected for $p\ne2$. A direct computation in $E_6(2)$
shows that $C_H(x)$ is also connected in characteristic~2. 
\par
In the finite group over $\FF_q$, the lengths $(q^4+1)(q^{12}-1)$ and
$q^3(q^4+1)(q^{12}-1)$ of these two orbits add up to the total number of long
root elements. We conclude as in the previous case. 
\end{proof}

We now turn to the proof of Theorem~\ref{thm:main} from the introduction.
One step is given by the following result:

\begin{lem}   \label{lem:inBorel}
 Let $G$ be an almost simple algebraic group over an algebraically closed
 field, and let $C,D$ be $G^{\circ}$-classes.  If every
 $(x_1,x_2)\in C\times D$ normalizes some Borel subgroup of $G^\circ$
 then $CD$ is a finite union of $G^{\circ}$-conjugacy classes.
\end{lem}

\begin{proof}
Let $(x_1,x_2)\in C\times D$ normalize the Borel subgroup $B$ of $G^\circ$.
If $x_1$, say, is inner, then conjugates of $x_1$ inside $B$ have only
finitely many distinct semisimple parts, so the possible products with $x_2$
have only finitely many quasi-semisimple classes in their closures, whence we
get~(i) by \cite[Lemma~5.1]{GMT} (note the argument shows that the number
of possible quasi-semisimple classes is independent of the Borel subgroup).

So now let's assume that $C$ and $D$ are both outer. We claim that this cannot
occur aside from the case that $C$ and $D$ are in distinct cosets of
outer involutions (whence $G = D_4(k).\fS_3$).

Note that under the assumption that every pair in $C \times D$ normalizes
a Borel subgroup, the conclusion holds if and only if it holds for
$C' \times D'$ where $C'$ and $D'$ are the closed quasi-semisimple classes
in $\bar{C}$ and $\bar{D}$ respectively.  So it suffices to assume that $C$
and $D$ are both quasi-semisimple. Note that our condition is also true for
powers of $C$ and $D$.

By choosing a parabolic subgroup that is invariant under the graph automorphism
and such that it still induces a graph automorphism on its Levi subgroup,
we can reduce to the cases $A_1(k)^2$, $A_2(k)$ or $A_1(k)^3$ (the latter
in $D_4(k)$). It is straightforward in each of these cases to
check that there is not always a common Borel subgroup unless possibly
$C$ and $D$ are in distinct cosets of graph automorphisms of order $2$.
Moreover, in the latter case, if neither $C$ nor $D$ have order $2$, then one
sees (as in the proof of Proposition~\ref{prop:D4.S3}), there will not always
be a common invariant Borel subgroup. So we may assume that $C$ consists of
graph automorphisms of order~$2$. Similarly, if $D$ is not as given in
Proposition~\ref{prop:D4.S3}(1), there will not always be a common invariant
Borel subgroup.  If $D$ does satisfy Proposition~\ref{prop:D4.S3}(1), then
in fact there are only finitely many classes in $CD$.
\end{proof}

We are now ready to prove Theorem~\ref{thm:main}, which is part of the
following statement:

\begin{thm}   \label{thm:mainlong}
 Let $G$ be an almost simple algebraic group over an algebraically closed
 field, and let $C,D$ be $G^\circ$-classes in $G$. 
 The following are equivalent:
 \begin{enumerate}[\rm(i)]
  \item $CD$ is a finite union of $G^\circ$-conjugacy classes.
  \item The closure of $CD$ contains a unique quasi-semisimple conjugacy class.
  \item $C_{G^\circ}(x_1)\backslash G^\circ/C_{G^\circ}(x_2)$ is finite for all
   $(x_1,x_2)\in C\times D$.
  \item $G^\circ$ has finitely many orbits on $C \times D$.
  \item $\langle x_1,x_2 \rangle$ normalizes some Borel subgroup of
   $G^\circ$ for every $(x_1,x_2)\in C\times D$.
  \item $C_{G^\circ}(x_1)C_{G^\circ}(x_2)$ is dense in $G^\circ$ for some
   $(x_1,x_2)\in C\times D$.
 \end{enumerate}
 Moreover, these conditions imply, and if $G \ne D_4.\fS_3$ are all
 equivalent to
 \begin{enumerate}[\rm(i)]
  \item[\rm(vii)] $[C,D]$ is a finite union of $G^\circ$  conjugacy classes.
 \end{enumerate}
\end{thm}

\begin{proof}
Statements (i) and (ii) are equivalent by Corollary~\ref{cor:unique} and
clearly (iii) and (iv) are equivalent. Furthermore, (iii) implies (vi).
By \cite[Lemma 5.3]{GMT} we have that (vi) implies (ii).

It is shown in Lemma~\ref{lem:inBorel} that (v) implies (i).
The equivalence of (i) and (vii) is the statement of Theorem~\ref{thm:commut}.

Thus it remains to show that (i) implies (iii) and (v), which we do by
going through the cases in Theorem~\ref{thm:mainB}. For the cases in~(1),
this is \cite[Cor.~5.14]{GMT}. The unipotent examples in~(2) are those from
Theorem~\ref{thm:mainA}(2)--(5). Here, it is shown in
Examples~\ref{ex:graph-trans}, \ref{exmp:D4}, \ref{exmp:E6}, and in
\cite[Ex.~6.6]{GMT} that $\langle x_1,x_2^g \rangle$ is a unipotent subgroup
and thus~(v) holds. Furthermore, we have~(iii) by Example~\ref{ex:graph-trans}
and Propositions~\ref{prop:orb1} and~\ref{prop:orb2}. 
\par
Now let's turn to the non-unipotent examples in
Theorem~\ref{thm:mainB}(e)--(k). For case~(e), there is just one orbit
of $G$ on $x_1^G \times x_2^G$ by \cite[Ex.~7.2]{GMT}, and clearly we can
find a pair in the normalizer of some Borel subgroup, whence (iii) and (v)
hold. In case~(f), we have just two orbits by Example~\ref{ex:graph-trans}.
Since clearly the graph automorphism $\sg$ cannot centralize all transvections
in a $\sg$-stable Borel subgroup of $G^\circ$, we get representatives of both
orbits in the normalizer.

In case~(g), by Example~\ref{exmp:GO} there's again just a single orbit.
In case~(h), by Example~\ref{exmp:GOtrans} both elements always lie in the
normalizer of some parabolic subgroup, and $x_1$ in its unipotent radical,
so we get~(v), while (iii) is in Proposition~\ref{prop:orb1}. Finally,
in cases~(i) and~(j), by Example~\ref{exmp:D4andE6} there is a dense pair,
which lies inside a Borel subgroup, and (iii) was shown in
Proposition~\ref{prop:orb2}. If $G=D_4.\fS_3$, we apply
Proposition~\ref{prop:D4.S3} and Lemma~\ref{lem:D4.S3}.
\end{proof}

The same result holds for $G$-classes as well (but we can include (vii) without
any restrictions as all the cosets of outer involutions are conjugate).  

Note that all in cases we see from the above arguments that $G$ either has
infinitely many orbits or at most $3$ orbits on $C_1 \times C_2$.  This was
shown in the connected case in \cite{GMT} aside from two cases where a bound
of $4$ was given. We show that $3$ is an upper bound in these cases  as well
in the next two results.

\begin{prop}   \label{lem:SPcent}
 Let $k$ be an algebraically closed field of characteristic $2$.
 Let $G=\Sp_{2n}(k)=\Sp(V)$. Let $x \in G$ be an involution with $(xv,v)=0$
 for all $v \in V$. Then $C_G(x)$ has at most $3$ orbits on the nonzero
 vectors of $V$.
\end{prop}  

\begin{proof}
Note that $W:=[x,V]$ is a totally singular $m$-space for some $m \le n$.  
Let $X=[x,V]^{\perp}$. Let $P$ be the stabilizer of $W$.  So $P=QL$ is a
maximal parabolic subgroup of $G$ with unipotent radical $Q$ and Levi subgroup
$L \cong \GL(W) \times \Sp(X/W)$. Note that $x \in Z(Q)$ and so we see that
$C_G(x) = QJ$ where $J \cong \Sp(W) \times \Sp(X/W)$.
The result now follows easily by noting that $J$ acts transitively on the
nonzero elements of $W$, $X/W$ and $V/X$ and that if $v \in X \setminus{W}$,
then $Qv = v +W$ and if $v \in V \setminus{X}$, then $Qv = v + X$
(if $W \ne X$, the three orbits are $W \setminus\{0\}$, $X \setminus{W}$ and
$V \setminus{X}$ while if $W=X$, there are two orbits).
\end{proof} 

\begin{prop}   \label{lem:SOcent}
 Let $k$ be an algebraically closed field of characteristic $p \ne 2$.
 Let $G=\SO_n(k)=\SO(V)$. Let $x \in G$ be a unipotent element
 with all Jordan blocks of size at most $2$. 
 Then $C=C_G(x)$ has at most $3$ orbits on nondegenerate $1$-spaces of $V$.
\end{prop}  

\begin{proof}
First suppose that $n$ is a multiple of $4$ and all Jordan blocks of $x$
have size $2$.  Then $W:=[x,V]=C_V(x)$ is a maximal totally singular subspace
and $C = Q\Sp(W)$ where $Q$ is the unipotent radical of the stabilizer of
$W$.  Then $C$ is transitive on all cosets $v + W$ with $v$ outside $W$ and
$Q$ is transitive on all vectors in $v+W$ of a given norm, whence we see
that $C$ has a single orbit on nondegenerate $1$-spaces (this is also clear
from the proof of Proposition~\ref{prop:orb1}). Thus, we see that $C$
has two orbits on nonzero singular vectors (they are either in $W$ or not).    

In the general case, we can write $V=V_0 \perp W$ where $V_0$ is a
nondegenerate space with $x$ trivial on $V_0$ and all Jordan blocks of $x$
have size $2$ on $W$.  Let $v$ be a nonsingular vector in $V$ and write
$v = v_0 + w$ with $v_0 \in V$ and $w \in W$.  By the argument in the first
paragraph,  we can find $x$-invariant nondegenerate subspaces
$V_1 \subset V_0$ and $W_1 \subset W$ with $\dim V_1 \le 2$ and
$\dim W_1 \le 4$ and $c \in C$ with $cv \in V_1 + W_1$. Thus, we may assume
that $n \le 6$ (and clearly $n \ge 4$).  It is straightforward to check the
result in these cases.
\end{proof}  

This completes the proof of Corollary~\ref{cor:finite}.

\section{Cosets of order 2 in $D_4(k).\fS_3$}   \label{sec:D4.S3}

Here, we discuss the situation left open in conclusion~(k) of
Theorem~\ref{thm:mainB}.

Let $k$ be an algebraically closed field of characteristic $p\ge 0$ and
$G=D_4(k).\fS_3$ (of adjoint type), the extension of a simple algebraic group
of type $D_4$ by the full group of graph automorphisms $\fS_3$. 
Let $x, y \in G$ be outer involutions in two
different cosets of $G^\circ$.  Moreover, we assume that $x$ and $y$
correspond to reflections in some projective $8$-dimensional representation of
$\langle G^\circ,x\rangle$, resp.~$\langle G^\circ,y\rangle$. Note
that $G^{\circ} = C_{G^\circ}(x)C_{G^\circ}(y)$ whence $G^\circ$ is transitive
on such pairs. Set $z=xy$.  Thus $z$ is a graph automorphism of $G^\circ$ of
order $3$ with centralizer $C_{G^\circ}(z) \cong G_2(k)$.

We start with some examples. In all of these $C$ denotes the $G^\circ$-class
of $x$ and $D$ will be a class in the coset $G^\circ y$. For convenience we
recall the outer unipotent conjugacy classes of $\GO_8(k)$,
with $k$ algebraically closed of characteristic~2 (see \cite[p.236]{Spa} and
\cite[Tab.~10]{MDec}, and the unipotent classes of $G_2(k)$ (see
\cite[Tab.~B, Tab.~1]{Law}).

\begin{table}[htbp]
 \caption{Outer unipotent classes of $\GO_8(k)$, $\Char(k)=2$}
  \label{tab:outD4}
\[\begin{array}{|l|cccc|}
\hline
     x& C(x)/R_u(C(x))& \dim R_u(C(x))& \dim C(x)& A(u)\cr
\hline
    2.1^6&   C_3& 0& 21& 1\cr
  2^3.1^2& A_1^2& 7& 13& 1\cr
    4.1^4& A_1^2& 5& 11& Z_2\cr
    3^2.2&   A_1& 6&  9& 1\cr
  4.2_0^2&   A_1& 6&  9& 1\cr
    4.2^2&     1& 7&  7& 1\cr
    6.1^2&   T_1& 4&  5& Z_2\cr
        8&     1& 3&  3& 1\cr
\hline
\end{array}\]
\end{table}

\begin{table}[htbp]
 \caption{Unipotent classes of $G_2(k)$, $\Char(k)\ne3$}
  \label{tab:G2}
\[\begin{array}{|l|cccc|rr|}
\hline
     x& C(x)/R_u(C(x))& \dim R_u(C(x))& \dim C(x)& A(u)& \span{\text{in }D_4(k)}\quad\cr
     & & & & & (p>3)& (p=2)\cr
\hline
          1&   G_2& 0& 14&     1&      1^8&        1^8\cr
        A_1&   A_1& 5&  8&     1&  2^2.1^4&  2_0^2.1^4\cr
 \tilde A_1&   A_1& 3&  6&     1&  3.2^2.1&        2^4\cr
   G_2(a_1)&     1& 4&  4& \fS_3&  3^2.1^2&    3^2.1^2\cr
        G_2&     1& 2&  2& Z_{(2,p)}&  7.1&        6.2\cr
\hline
\end{array}\]
\end{table}

\begin{table}[htbp]
 \caption{Unipotent classes of $G_2(k)$, $\Char(k)=3$}
  \label{tab:G2p=3}
\[\begin{array}{|l|cccc|r|}
\hline
     x& C(x)/R_u(C(x))& \dim R_u(C(x))& \dim C(x)& A(u)& \text{in }D_4(k)\cr
\hline
          1&   G_2& 0& 14&     1& 1^8\cr
        A_1&   A_1& 5&  8&     1& 2^2.1^4\cr
 \tilde A_1&   A_1& 5&  8&     1& 3.2^2.1\cr
 \tilde A_1^{(3)}& 1& 6& 6&    1& 3.2^2.1\cr
   G_2(a_1)&     1& 4&  4&   Z_2& 3^2.1^2\cr
        G_2&     1& 2&  2&   Z_3& 7.1\cr
\hline
\end{array}\]
\end{table}


\begin{exmp}   \label{exmp:D4.S3a}
Let $D$ consist of elements conjugate to $ys = sy$ where $s\in C_{G^\circ}(y)$
is a semisimple element with at most $2$ non-trivial eigenvalues. First
assume that $s$ has a non-trivial eigenvalue $a$ of order prime to~$6$.
Note that $\dim C = 7$ and that $\dim D = 17$. Also note that we can arrange
that $xy$ is conjugate to $zt$ where $t\in C_{G^\circ}(z)$ with $t$
semisimple and of order prime to $6$.
Thus $\dim C_{G^\circ}(zt) = \dim C_{G_2(k)}(t^3) \le 4$ and it follows
that $\dim (xy)^{G^\circ} \ge 24$, whence we have equality.
Thus, $CD=(xy)^{G^\circ}$ (since the right hand side is closed).
Indeed the same argument shows that for any $(x',y')\in C \times D$,
$\dim (C_{G^\circ}(x') \cap C_{G^\circ}(y'))\le \dim C_{G^{\circ}}(xy) = 4$,
so $G^{\circ} = C_{G^\circ}(x)C_{G^\circ}(y)$.
If $a \in k^\times$ is arbitrary, the centralizer will only get bigger
whence we still have the same factorization.
\end{exmp}

\begin{exmp}   \label{exmp:D4.S3b}
Let $p\ne2$ and $D$ consist of elements conjugate to $yu=uy$ where
$u\in C_{G^\circ}(y)$ is unipotent with Jordan block lengths $2^2.1^4$ or
$3.1^5$.\par
First consider the case that $u$ is a long root element. Note that every
$G^\circ$-orbit in $C \times D$ contains a representative $(x,yu)$ where $u$
is a long root element commuting with $y$ (since $G^\circ$ is transitive on
$x^{G^\circ} \times y^{G^\circ}$). By Proposition~\ref{prop:orb2},
$C_{G^\circ}(z)$ has five orbits on long root elements in $G$.
It is easy to see that only $2$ of these orbits are contained in $C_G(y)$.
Thus, $G^{\circ}$ has exactly $2$ orbits on $C \times D$ (depending upon
whether $u$ centralizes $z$ or not).

There is a subgroup $A_1(k)^4$ of $G^\circ$ normalized by $x,y$. We may assume
that $y$ permutes the last two copies and $x$ permutes the middle two copies
(so both centralize the first copy).  If we take $u = (a_1, a_2, 1, 1)$
where $a_i$ are nontrivial unipotent elements of $A_1(k)$, then
$u$ has a single Jordan block of size $3$ but $zu$ is conjugate to
$zv=vz$ where $v$ has two Jordan blocks of size $3$ (and two trivial Jordan
blocks). It follows that $v$ lies in the conjugacy class $G_2(a_1)$ in
$C_{G^\circ}(z)$. Thus, $\dim (zv)^{G^\circ} = 24$ and so $G^\circ$ has a
dense orbit on $C \times D$ (consisting of those pairs whose product is in
the class of $zv$).
\end{exmp}

\begin{exmp}   \label{exmp:D4.S3c}
Let $p=2$ and $D$ consist of $2$-elements with Jordan block lengths
$2^3.1^2$, $4.1^4$ or $3^2.2$ in the natural 8-dimensional representation
of $\GO_8(k)$. An explicit calculation in the permutation representation of
$\SO_8^+(2).\fS_3$ of degree~3510 (as maximal subgroup of $Fi_{22}$) shows that
the respective products contain elements of the form $zu$, with $z=xy$ a
graph automorphism of order~3 and $u\in C_{G^\circ}(z)\cong G_2(k)$ unipotent
in the $G_2(k)$-class $\tilde A_1,G_2(a_1),G_2$ respectively (see
Table~\ref{tab:G2}). Thus the corresponding classes $E=[zu]$ have dimensions
$22, 24, 26$ respectively. As $\dim C+\dim D$ equals $7+15,7+17,7+19$ in the
three cases, we conclude by Lemma~\ref{lem:dense} that $E$ is dense in $CD$
and that $CD$ consists of finitely many $G^\circ$ classes.
\end{exmp}

We now show that the above are the only examples:

\begin{prop}   \label{prop:D4.S3}
 Let $C,D$ be $G^\circ$-orbits in $xG^{\circ}$ and $yG^{\circ}$. Then $CD$
 is a finite union of $G^\circ$-orbits if and only if (after reordering if
 necessary) $C$ consists of involutions conjugate (in $G$) to $x$ and one of
 \begin{enumerate}
  \item[\rm(1)] $D$ consists of elements conjugate to $yt = ty $ where
   $t \in G^\circ$ is a semisimple element with at most $2$ non-trivial
   eigenvalues;
  \item[\rm(2)] $p\ne2$, $D$ consists of elements conjugate to $yu=uy$ where
   $u$ is unipotent with Jordan block lengths $2^2.1^4$ or $3.1^5$; or
  \item[\rm(3)] $p=2$ and $D$ consists of $2$-elements with Jordan block
   lengths $2^3.1^2$, $4.1^4$ or $3^2.2$ in the natural 8-dimensional
   representation of $\GO_8(k)$.
 \end{enumerate}
\end{prop}

\begin{proof}
We have already seen in Examples~\ref{exmp:D4.S3a}--\ref{exmp:D4.S3c} that
all cases~(1)--(3) really occur. So we need to show that these are the only
possibilities with $CD$ containing only finitely many $G^\circ$-orbits.
So assume that $CD$ consists of finitely many $G^\circ$-orbits or equivalently
that all quasi-semisimple elements in the closure of $CD$ are conjugate.

First consider the case that $p \ne 2$ and  $C$ and $D$ are both involutions but
neither are reflections. Note that this prescribes $C$ and $D$ uniquely.
Thus, for any pair of involutions $s, t \in C_{G^\circ}(z)$,
$(xs,yt) \in C \times D$ and $xsyt=z(st)$.  As $s,t$ run over all such
involutions, $st$ intersects infinitely many $G_2(k)$ classes
(since we can take $s, t$ to be involutions in $G_2(k)$), whence
$CD$ consists of infinitely many classes.

If $p=2$ and $C$ and $D$ both consist of $2$-elements, then every unipotent
class is represented in $\Aut(D_4(2))$ and one checks directly
using GAP that the answer is as stated in (3) of the assertion.

So next assume that $D$ does not consist of involutions and that
$C$ and $D$ are both quasi-semisimple classes.
Then we choose $(x',y') \in C \times D$ which normalize a parabolic subgroup
$P=QL$ with $P$ the normalizer of a long root subgroup, $L$ is a Levi subgroup
and $Q$ the unipotent radical of $P$.
Note that  $L = TA_1(k)^3$ where $T$ is a central $1$-dimensional
torus of $L$.  We can assume that $x$ permutes the first
two copies of $A_1(k)$ and $y$ the last two.

Since $C$ and $D$ consist of quasi-semisimple elements, we can choose
$x',y' $ actually normalizing $L$.  It is straightforward to see that
(aside from the cases allowed in the conclusion) $(C \cap L)(D \cap L)$
intersects infinitely many classes of $L.\fS_3$ whence $CD$ contains
infinitely many $G^\circ$-orbits by Proposition~\ref{prop:hit}.

Thus, we may assume that all quasi-semisimple elements in the closure of $x'$
are conjugate to $x$ and that the quasi-semisimple elements in the closure of
$y'$ are as given in (1) or (2).  Moreover, at least one of $x', y'$ is not
quasi-semisimple. 

\medskip
Case 1. $x'$ is quasi-semisimple and $y$ does not lie in the closure of $y'$.

\noindent
By passing to closures, we may assume that $y'=y_1y = yy_1$ where $y_1 = ut$
with $t$ semisimple having precisely two nontrivial eigenvalues and $u$ is a
long root element.  Note that such a class has a representative in $L$ and
one easily checks that there are infinitely many classes represented there.

\medskip
Case 2. $x'$ is quasi-semisimple and $y$ lies in the closure of $y'$.

\noindent
In this case, we can assume that $p \ne 2$ (because of the remarks above about
pairs of $2$-elements). By passing to closures (and noting that we may assume
the unipotent part of $y'$ is not a long root element nor has just one single
Jordan block of size $3$), we may assume that $y'=yu$ where $u$ has a single
Jordan  block of size $3$ and two Jordan blocks of size $2$.
Again, we can see such an element in $L$ and obtain a contradiction. 

\medskip
Case 3.  Neither $x'$ nor $y'$ are quasi-semisimple.

\noindent
Passing to closures, we may assume that the unipotent parts of $x'$ and
$y'$ are long root elements and again we can see such elements in $L$
and derive a contradiction.

By symmetry, we have covered all cases and proved the result. 
\end{proof} 

\begin{lem} \label{lem:D4.S3}
 In the situation of Proposition~\ref{prop:D4.S3}, $G^\circ$ acts transitively
 on $C \times D$ in case~(1); it has $2$ orbits in both situations of case (2);
 and there are 2, 3,  respectively~2 orbits in case (3).  Moreover,
 every $(c,d) \in C \times D$ normalizes a Borel subgroup of $G^\circ$. 

\end{lem}

\begin{proof}
The case~(1) was already argued in Example~\ref{exmp:D4.S3a}, and similarly
for the first pair of classes in case~(2), in Example~\ref{exmp:D4.S3b}. If
$u$ has Jordan blocks of lengths $3.1^5$ in case~(2), we need to count
orbits of $G_2(k)$ on the class of $u$ in $\SO_7(k)$.

Note that if $u\in \SO_7(k)$ is such a unipotent element, it has a unique
1-dimensional invariant (singular) subspace (namely the image of $(u-1)^2$)
and then $u$ lies in the unipotent radical of the stabilizer $P_1$ of that
subspace (corresponding to nonsingular vectors in the unipotent radical).

Note that $G_2(k)$ acts transitively on singular 1-spaces (this is well known,
but note that in the 7-dimensional representation, the stabilizer of a highest
weight space is a maximal parabolic subgroup $P$ of codimension~5, thus
$G_2/P$ is a 5-dimensional projective variety contained in $\SO_7(k)/P_1$
which is also projective of dimension~5, so they are equal). So we can just
count orbits of $P$ on the unipotent radical $Q$ of $P_1$.
Now $Q$ is 5-dimensional and abelian and $P_1/Q$ acts on this as $T\SO_5(k)$,
and $P/(Q \cap P) = TA_1U_3$ acts on this so that $Q = 2/1/2$ as a module for
$P$. Because of the torus it is enough to count orbits on nondegenerate
1-spaces, whence it is easy to see that $P'$ has two such orbits acting on
$Q$ (namely those spaces in the orthogonal complement of the 2 and those
outside).
\par

Finally, in case~(3), first assume that $D$ contains unipotent elements
with Jordan type $2^3.1^2$. As we saw in Example~\ref{exmp:D4.S3c}, one orbit
consists of pairs $(c,d)$ such that $cd$ lies in class $\tilde A_1$ of
a subgroup $G_2(k)$, with connected centralizer of dimension~6. A second
orbit is given by pairs $(c,d)$ with product in class $A_1$ in $G_2(k)$, and
connected centralizer of dimension~8. Adding up the lengths of these two
orbits over $\FF_{2^f}$ shows that there no other orbits.

Next, assume that $D$ contains unipotent elements with Jordan type $4.1^4$.
Any $d\in D$ can be written as $yw$ where $y$ is
a reflection and $w$ is a long root element in $G^\circ$ (not centralizing $y$
for sure). This is not unique, but given any $(c, d)\in C \times D$,
we can conjugate and assume $c = x$ and $d = yw$ as above.
Now $G_2(k)$ centralizes $x$ and $y$ and so the number of orbits
of $G$ on $C \times D$ is at most the number of orbits of $G_2(k)$ on the set
of long root elements $w$ such that $yw$ has Jordan type $4.1^4$.

Since $G_2(k)$ has 5 orbits on long root elements and at least two of them do
not have the property that $yw$ has order $4$ (namely the ones in $G_2(k)$ and
also the ones in $\Sp_6(k)$), there are at most 3 orbits. In fact, we claim
that there are $3$ orbits. Two of the orbits are interchanged by $y$ and we
can see such a $w$ in $A_1(k) \times A_1(k)$ where $y$ is interchanging the 
$2$ factors and $w$ is trivial in one of the factors. Clearly $yw$ has order
$4$ and has a $5$-dimensional fixed spaces, so $yw$ has the correct Jordan
type. Since $y$ interchanges those two orbits, we see that there are at least
$2$ orbits.  However, neither of these is the dense orbit (e.g., they have the
same dimension) and so there must be a third orbit (alternatively, one
can compute over $\FF_2$). 
\par

Finally consider the case that $D$ consists of unipotent elements with
Jordan type $3^2.2$. 

Computing over $\FF_2$, we see that $CD$ contains the conjugacy classes
of elements $zu=uz$ where $z$ is a graph automorphism of order $3$
with centralizer $G_2(k)$ and $u$ is either a regular unipotent element of
$G_2(k)$ or is in the conjugacy class $G_2(a_1)$.  

First consider a pair $(c,d) \in C \times D$ with $cd=zu=uz$ where $u$ is a
regular unipotent element in $G_2(k)$.  It follows that $J$, the
centralizer of $\langle c, d \rangle$ is contained in the centralizer of $zu$
which has connected component a 2-dimensional unipotent group and has two
components.  Since $\dim C + \dim D = 26$, it follows that $\dim J \ge 2$,
whence $J^\circ$ is a $2$-dimensional unipotent group.  Moreover, this must
correspond to the dense orbit of $G^\circ$ and so it is unique. 
Computing such triples over $\FF_2$, we see that there are two orbits 
each of size $|G^\circ(2)|/4/2$.  It follows by Lang's theorem 
that $[J:J^{\circ}]=2$ and so there is a single $G^\circ$-orbit splitting
into two $G(q)$-orbits each of size $(1/2)|G^\circ(q)|/q^2$.   

It follows by counting that the dimension of the complement of the dense orbit
in  $C \times D$ is $24$ and moreover, there is at most $1$ orbit
of that dimension. Let $H$ be the centralizer of a $(c,d) \in C \times D$
with product $zu=uz$ where $u \in G_2(a_1)$.  The centralizer $H$ of $zu$
in $G^\circ$ has connected component a $4$-dimensional unipotent group
with $H/H^{\circ} \cong \fS_3$ (see Table~\ref{tab:G2}).  We see that over
$\FF_2$ this orbit breaks
up into three orbits of sizes  $e|G^\circ(2)|/2^4$ withe $e=1/2,1/3$ and $1/6$.
It follows that $H$ is the centralizer of $\langle c, d\rangle$ and this orbit
breaks up into $3$ orbits over $\FF_q$ of sizes 
$e|G^\circ(q)|/q^4$ with $e=1/2,1/3$ and $1/6$.  Thus, the number
of $\FF_q$ points in this orbit of the algebraic group is $|G^\circ(q)|/q^4$.

Since $|C(q)|D(q)|=|G^\circ(q)|/q^4 + |G^\circ(q)|/q^2$, it follows that there
are no further orbits and so $G^\circ$ has exactly $2$ orbits on
$C \times D$. 

The last assertion follows by noting that in the examples we produce
$(c,d) \in C \times D$ in the dense orbit which normalize a Borel subgroup.
Since normalizing a Borel subgroup is a closed condition (as $G/B$
is a projective variety), this implies the result. 
\end{proof}

We next consider commutators of $G^\circ$-orbits in our situation. Here, we
have the following curious situation:

\begin{exmp}   \label{exmp:commD4}
As above let $x,y\in G=D_4(k).\fS_3$ be quasi-central in distinct cosets of
order~2 modulo $H:=G^\circ$. Then for any $g\in H$ we have
$$[x,yg] = x (yg)^{-1} x yg  =  x g^{-1}(yxy)g \in x^H t^H =(xt)^H$$
where $t = yxy$ is a quasi-central element (reflection) in the third such
coset. So letting $C = x^H$ and $D=yH$, we see that $[C,D]$ is a single
$H$-orbit (namely the graph automorphisms of order~3 with centralizer $G_2(k)$).
In particular, $G$ certainly need not have finitely many orbits on
$C \times D$. Hence, the analogue of Theorem~\ref{thm:mainlong}(vii) fails in
this situation.
\end{exmp}

\begin{prop}   \label{prop:commD4}
 Let $G=D_4(k).\fS_3$ and $C,D$ two $G^\circ$-orbits in distinct cosets of
 $G^\circ$ of order~2. Then $[C,D]$ is the union of finitely many
 $G^\circ$-orbits if and only if (up to order) $C$ consists of reflections.
\end{prop}

\begin{proof}
We have seen in Example~\ref{exmp:commD4} that if $C$ consists
of reflections and $D$ is any class, then $[C,D] = z^{G^\circ}$ where $z$ is
a graph automorphism of order~3 with centralizer $G_2(k)$.
Thus, by taking closures we may assume that one of the following holds:\\
(1)  $C$ and $D$ are both quasi-semisimple;\\
(2)  $C$ and $D$ are both reflections times  long root elements; or\\
(3)  $C$ is quasi-semisimple and $D$ is as in (2).

We can see all these cases in $A_1^4$ and get a contradiction
(assuming that neither $C$ nor $D$ consists of reflections).
\end{proof}

Here is another corollary (which could be proved more directly):

\begin{cor}   \label{cor:triple}
 Let $G$ be an almost simple algebraic group. Suppose that $C_1,C_2,C_3$ are
 nontrivial $G^\circ$-orbits. Then $C_1C_2C_3$ consists of infinitely many
 $G^\circ$-orbits unless (up to order) $C_1$ and $C_2$ are reflections in
 different cosets modulo $G^\circ$ and $C_3$ consists of long root elements.
\end{cor}

\begin{proof}
Observe that for $C_1$ and $C_2$ reflections in different cosets and $C_3$
long root elements, $C_1C_2 = D:=z^{G^\circ}$ and so indeed $C_1C_2C_3 = DC_3$
is a finite union of classes by Theorem~\ref{thm:mainB}. \par
In any case other than $D_4(k).\fS_3$ with an outer automorphism,
two of the $C_i$ must either be inner or outer and so already $C_iC_j$ consists
of infinitely many inner classes again by Theorem~\ref{thm:mainB} (and by
working in the normalizer of a torus, infinitely many classes times $C_k$ gives
infinitely many classes).

If $G$ is connected, by passing to closures, we may assume that each $C_i$ is
either semisimple or unipotent.  We know that two semisimple elements do not
work and so at least two of the classes must be unipotent. By taking closures,
they can be taken to be root elements (and so for roots of distinct lengths).
So we reduce to the rank~2 case, i.e. $B_2(k)$ (characteristic~2) or $G_2(k)$
(characteristic~3) by \cite{GMT}. But then the product of the two unipotent
classes contains a regular unipotent element and that is not one of our
examples.
\end{proof}

\section{Infinite Fields}   \label{sec:compact}

Let's also note the following easy consequence of our results in \cite{GMT}:

\begin{cor}
 If $G$ is a simple compact (real) Lie group and $C$ and $D$ are non-central
 conjugacy classes of $G$, then $CD$ is an infinite union of classes.
\end{cor}

\begin{proof}
It follows easily from the fact  that $G$ is Zariski dense in $\bar G$, the
complexification of $G$, that any class is dense in its closure $\bar C$
in $\bar G$. Now if $CD$ is a
finite union of classes, then taking closures shows that $\bar C \bar D$ is
a finite union of $\bar G$-classes as well, and as $C$ and $D$ are
semisimple classes, by \cite[Thm.~1.1]{GMT} this cannot happen.
\end{proof}

For the disconnected case, we would get only the semisimple cases
(e.g.~in type $A$). The previous result extends to infinite fields
with $G(K)$ an anisotropic simple group (i.e., containing no nontrivial
unipotent elements).   

The same ideas give the following result.  

\begin{cor}
 Let $K$ be an infinite field and $G(K)$ some form of a simple
 algebraic group $G$ over $K$. Let $C\subset G(K)$ be a non-central conjugacy
 class. Then $CC$ is an infinite union of classes.
\end{cor}

\begin{proof}
Let $\bar C$ denote the Zariski closure of $C$. Then the semisimple part
of elements in $\bar C$ is unique, since in any rational faithful
representation, elements in $\bar C$ have the same characteristic polynomial
as those in $C$, as $G(K)$ is Zariski dense in $G$ (see
e.g.~\cite[Cor.~13.3.9]{Spr}). Thus, if $CC$ consists of finitely many classes,
so does $\bar C\bar C$, which is not possible by \cite[Thm.~1.1]{GMT}.
\end{proof}


\end{document}